\documentclass[11pt,letterpaper]{amsart}

\usepackage{a4,amsmath,amssymb,amscd,verbatim,bbm,mathrsfs}

\newtheorem{theorem}{Theorem}
\newtheorem{proposition}[theorem]{Proposition}
\newtheorem{corollary}[theorem]{Corollary}
\newtheorem{lemma}[theorem]{Lemma}

\theoremstyle{definition}

\newtheorem{remark}[theorem]{Remark}

\newtheorem{example}[theorem]{Example}

\numberwithin{equation}{section}
\numberwithin{theorem}{section}

\def\<{\langle}
\def\>{\rangle}
\def\g{\hat{\rho}}
\def\o{\omega}
\def\l{\lambda}

\def\N{\mathbb{N}}

\def\R{\mathbb{R}}
\def\H{\mathbb{H}}

\def\P{\mathcal{P}}
\def\O{O}%{\mathcal{O}}
\def\lk{\mathrm{lk}}

\def\PSL{\operatorname{PSL}}
\def\curl{\operatorname{curl}}

\def\PSL{\operatorname{PSL}}

\begin{document}
\bibliographystyle{plain} \title[Helicity and linking]% \lebn]
{Helicity, linking and the distribution of null-homologous periodic orbits for Anosov flows} 

\author{Solly Coles} 
\address{Mathematics Institute, University of Warwick,
Coventry CV4 7AL, U.K.}
\email{Solly.Coles@warwick.ac.uk}

\author{Richard Sharp} 
\address{Mathematics Institute, University of Warwick,
Coventry CV4 7AL, U.K.}
\email{R.J.Sharp@warwick.ac.uk}

\thanks{\copyright 2022. This work is licensed under a CC BY license. Solly Coles was
supported by the UK Engineering and Physical Sciences Research Council.}

\keywords{}

%\date{\today}

\begin{abstract}
This paper concerns connections between dynamical systems, knots and helicity of vector fields.
For a divergence-free vector field on a closed $3$-manifold that generates an Anosov flow,
we show that the helicity of the vector field may be recovered as the limit of appropriately weighted averages of linking numbers of periodic orbits, regarded as knots. This complements a classical result of
Arnold and Vogel that, when the manifold is a real homology $3$-sphere,
the helicity may be obtained as the limit of the normalised linking numbers of typical pairs of long trajectories.
We also obtain results on the asymptotic distribution of weighted averages of null-homologous periodic orbits.
\end{abstract}

\maketitle

%
%
%
%
%
%
%
%
%
%%%%%%%%%%%%%%%%%%%%%%%%%%%%%%%%%%%%%%%%%%%%%%%%%%%%%%%%%%%%%%%%%%%%%%%%%%%%%

\section{Introduction}\label{in}
This paper concerns connections between dynamical systems, knots and helicity of vector fields.
More specifically, for a divergence-free vector field on a closed $3$-manifold that generates an Anosov flow,
we show that the helicity of the vector field may be recovered as the limit of appropriately weighted averages of linking numbers of periodic orbits, regarded as knots. This complements a classical result of
Arnold (whose proof was completed by Vogel) that, when the manifold is a real homology $3$-sphere,
the helicity may be obtained as the limit of the normalised linking numbers of typical pairs of long trajectories.

We will now outline the setting more precisely. For good overviews, 
see the book of Arnold and Khesin
\cite{AK} (which restricts to vector fields in regions of $\mathbb R^3$) and the beautiful 2006 ICM survey by Ghys \cite{ghys-icm}.
Let $M$ be a connected and oriented smooth closed $3$-manifold 
with a volume form $\Omega$.
(We recall that closed
here means compact and without boundary.)
Let $X$ be a divergence-free vector field which generates a flow $X^t : M \to M$.
Then $X^t$ preserves the volume.
Let $i_X$ denote the interior product associated to the vector field $X$. If the $2$-form $i_X\Omega$ is exact then we say that $X$ is {\it null-homologous}. Note that this holds
automatically if $M$ is a real homology $3$-sphere (which is equivalent
to requiring that $H_1(M,\mathbb R) =\{0\}$).

If $X$ is null-homologous we can find a $1$-form $\alpha$ with $i_X \Omega = d\alpha$.
Then the {\it helicity} of $X$ is defined by 
\[
\mathcal{H}(X)= \int_M \alpha \wedge d\alpha
\]
and is independent of the choice of $\alpha$ (see Section \ref{subsec:helicity}).
Helicity was introduced by Woltjer \cite{woltjer}, Moreau \cite{moreau} and Moffat \cite{moffat}
and is an invariant (of volume-preserving diffeomorphisms)
which measures the amount of knottedness of flow orbits.

Now suppose that $M$ is a real homology $3$-sphere.
In this case, an alternative characterisation of $\mathcal H(X)$ was given by 
Arnold \cite{arnold},
with some gaps in the proof completed by Vogel \cite{vogel}.
To formulate this, one needs a set $\Sigma$ of ``short curves'' joining each pair of points in $M$.
In particular, for 
$x \in M$ and $t>0$, $\sigma_t(x) \in \Sigma$ will be a curve from 
$X^t(x)$ to $x$.
Let $m$ be the volume measure associated to $\Omega$
(normalised to be a probability measure). For $x,y \in M$, define
\[
\mathcal A(x,y) := \lim_{s,t \to \infty} \frac{1}{st} \mathrm{lk}(X^{[0,s]}(x) \cup
\sigma_s(x),X^{[0,t]}(y) \cup \sigma_t(y)),
\]
where, for two knots $\gamma,\gamma'$ in $M$, $\mathrm{lk}(\gamma,\gamma') \in \mathbb Q$
denotes their linking number.
(Linking numbers will be defined in Section \ref{section:linking} below.)
The limit exists for $(m \times m)$-almost every $(x,y) \in M \times M$ and we have
\[
\mathcal H(X) = \int \mathcal A(x,y) \, d(m \times m).
\]
Furthermore, if $X^t$ is ergodic with respect to $m$ then
we have
\[
\mathcal H(X) = \mathcal A(x,y) 
\]
for $(m \times m)$-almost every $(x,y) \in M \times M$.

Now suppose that $M$ is a Riemannian manifold and that $X^t : M \to M$ is an Anosov flow (see Section \ref{section:anosov} for the definition)
which is null-homologous. (Examples of null-homologous Anosov flows are
geodesic flows over negatively curved manifolds and, more generally, 
contact Anosov flows \cite{foulon-hasselblatt}.)
Anosov flows are chaotic and have a complicated orbit structure. However,
a well-known feature 
is that one can recover global invariants by averaging periodic orbit data. Our main result will be to show that the helicity $\mathcal H(X)$ may be recovered as a limit of weighted averages
of linking numbers of certain periodic orbits (regarded as knots).
This is very much inspired by results of Contreras \cite{contreras} about the linking of periodic orbits of hyperbolic flows on $S^3$.

Let $\mathcal P$ denote the set of prime periodic orbits for $X^t$ and let $\mathcal P(0)
\subset \mathcal P$ denote those which are trivial in $H_1(M,\mathbb R)$.
For $\gamma \in \mathcal P$, let $\ell(\gamma)$ denote its least period.
 For $T>0$,
let
\[
\mathcal P_{T} = \{\gamma \in \mathcal P
\hbox{ : } T-1<\ell(\gamma) \le T\},
\quad
\mathcal P_{T}(0) = \{\gamma \in \mathcal P(0)
\hbox{ : } T-1<\ell(\gamma) \le T\}.
\]
Since $\mathcal P_T(0)$ and $\mathcal P_{T+1}$ are disjoint, and $\mathcal P_T(0)$ consists of null-homologous orbits, the linking number of a
pair of periodic orbits from these two collections are well-defined. (The choices of intervals in which the least period lies
 is somewhat arbitrary but it is important that each pair consists of two distinct orbits; see the discussion in Remark \ref{remark:formulation} in Section \ref{section:per_orbits_helicity}.)

Periodic orbits will be weighted by the function $\varphi^u : M \to
\mathbb R$ defined by
\[
\varphi^u(x) = -\lim_{\epsilon \to 0} \frac{1}{\epsilon} \log |\det(DX^\epsilon|E_x^u)|,
\]
where $E_x^u$ is the fibre of the unstable bundle at $x$. Then, for $\gamma \in \mathcal P$,
the integral 
\[
\int_\gamma \varphi^u := \int_0^{\ell(\gamma)} \varphi^u(X^t(x)) \, dt
\] 
measures the expansion around $\gamma$. 
(When $M$ is a 3-manifold, the unstable bundle in one-dimensional and it is not necessary to take
a determinant in the definition of $\varphi^u$. However, we have equidistribution results below which are valid in arbitrary dimensions, so we give the general definition.)

We define $\varphi^u$-weighted average linking numbers
over the sets of orbits $\P_{T}(0)$ and $\mathcal P_{T+1}$ by
\[
\mathscr L_{\varphi^u}(T):=\frac{\displaystyle \sum_{\gamma\in \P_{T}(0),\gamma'\in \mathcal P_{T+1}}
\frac{\mathrm{lk}(\gamma,\gamma')}{\ell(\gamma) \ell(\gamma')}
\exp\left(\int_\gamma \varphi^u+\int_{\gamma'} \varphi^u\right)}
{\displaystyle \sum_{\gamma\in \P_{T}(0),\gamma' \in \mathcal P_{T+1}}
\exp\left(\int_\gamma \varphi^u+\int_{\gamma'} \varphi^u\right)}
\]

Our main result is the following.

\begin{theorem}\label{thm:hel-homfull-intro}
Let $X^t : M \to M$ be a null-homologous volume-preserving Anosov flow on a 
closed oriented $3$-manifold.
Then 
$$
\mathcal H(X)
= \lim_{T \to \infty} \mathscr L_{\varphi^{u}}(T).
$$ 
\end{theorem}

This is restated and proved as Theorem \ref{thm:hel-homfull} below.
If $M$ is a real homology $3$-sphere then all periodic orbits are null-homologous
and so we can replace the limit in Theorem \ref{thm:hel-homfull-intro} with a limit
of averages over all periodic orbits, as follows.

\begin{corollary}\label{cor:rh3s-intro}
Let $M$ be a real homology $3$-sphere and let $X^t : M \to M$
be a volume-preserving Anosov flow. Then
\[
\mathcal H(X)
=
\lim_{T \to \infty}
\frac{\displaystyle \sum_{\gamma\in \P_{T},\gamma'\in \mathcal P_{T+1}}
\frac{\mathrm{lk}(\gamma,\gamma')}{\ell(\gamma) \ell(\gamma')}
\exp\left(\int_\gamma \varphi^u+\int_{\gamma'} \varphi^u\right)}
{\displaystyle \sum_{\gamma\in \P_{T},\gamma' \in \mathcal P_{T+1}}
\exp\left(\int_\gamma \varphi^u+\int_{\gamma'} \varphi^u\right)}.
\]
\end{corollary}

A particular case of a real homology 3-spheres is provided by the unit tangent bundle
of a genus zero hyperbolic orbifold, and the geodesic flow is Anosov. In this case, $\varphi^u$ is 
a constant, $\varphi^u=-1$, and so, by Corollary \ref{cor:rh3s-intro}, we 
may obtain the helicity in the limit by weighting each closed geodesic $\gamma$ 
by $e^{-\ell(\gamma)}$. However, we 
may obtain the following unweighted result, proved as Theorem 
\ref{thm:genus-zero-orbifold} in Section 8.

\begin{theorem}\label{thm:genus-zero-orbifold-intro}
Let $X^t : M \to M$
be the geodesic flow over a genus zero hyperbolic orbifold. Then
\[
\mathcal H(X)
=
\lim_{T \to \infty}
\frac{1}{\#\mathcal P_T \ \#\mathcal P_{T+1}}
\sum_{\gamma\in \P_{T},\gamma'\in \mathcal P_{T+1}}
\frac{\mathrm{lk}(\gamma,\gamma')}{\ell(\gamma) \ell(\gamma')}.
\]
\end{theorem}

In order to prove these results, we need to understand the limiting behaviour of the weighted
orbital measures
\[
\left(\sum_{\gamma\in \P_{T}}
e^{\int_\gamma \varphi^u}\right)^{-1}
\sum_{\gamma\in \P_{T}}
e^{ \int_\gamma \varphi^u} \mu_\gamma
\]
and
\[
\left(\sum_{\gamma\in \P_{T}(0)}
e^{\int_\gamma \varphi^u}\right)^{-1}
\sum_{\gamma\in \P_{T}(0)}
e^{ \int_\gamma \varphi^u} \mu_\gamma,
\]
as $T \to \infty$, where $\mu_\gamma$ is defined by
\[
\int \psi \, d\mu_\gamma := \frac{1}{\ell(\gamma)} \int_\gamma \psi.
\]
By a result of Parry \cite{parry-cmp}, the first family of measures converges to the volume measure
$m$. (This is discussed further in Section \ref{section:equi_per_orbits}.)
The situation for null-homologous periodic orbits (or periodic orbits restricted to lie in any prescribed
homology class) is more delicate, and we believe of independent interest. The unweighted version (where $\varphi^u$ is replaced by zero)
was studied by Katsuda and Sunada \cite{KS} and Sharp \cite{Sh93}. Here we consider for the first time the weighted version. 
The next theorem gives the result we need to prove
Theorem \ref{thm:hel-homfull-intro}, where the weighting is by $\varphi^u$ and the limiting
measure is the volume.

\begin{theorem}\label{weighted-equi-nullhom-intro}
Let $X^t : M \to M$ be a null-homologous volume-preserving Anosov flow on a 
closed oriented $3$-manifold. Then the measures
\[
\left(\sum_{\gamma\in \P_{T}(0)}
e^{\int_\gamma \varphi^u}\right)^{-1}
\sum_{\gamma\in \P_{T}(0)}
e^{ \int_\gamma \varphi^u} \mu_\gamma,
\]
converge to $m$ in the weak$^*$ topology, as $T \to \infty$.
\end{theorem}

More generally, we can handle any H\"older continuous weighting $\varphi : M
\to \mathbb R$ and Anosov flows in arbitrary dimensions. 
We have the following result, which we restate more precisely later as
Theorem \ref{weighted-equi}.

\begin{theorem} \label{weighted-equi-intro}
Let $X^t : M \to M$ be a homologically full transitive Anosov flow. 
Let $\varphi : M \to \mathbb R$ be H\"older continuous.
Then the measures
\[
\left(\sum_{\gamma\in \P_{T}(0)}
e^{\int_\gamma \varphi}\right)^{-1}
\sum_{\gamma\in \P_{T}(0)}
e^{ \int_\gamma \varphi} \mu_\gamma,
\]
converge in the weak$^*$ topology, as $T \to \infty$,
and the limiting measure is characterised by a variational principle.
\end{theorem}

We end the introduction by outlining the structure of the paper. In Section \ref{section:anosov},
we define Anosov flows and introduce winding cycles and homologically full flows.
In Section \ref{section:pressure}, we recall some notions from thermodynamic formalism: entropy, pressure and the
special invariant measures known as equilibrium states.
In Section \ref{section:equi_per_orbits}, we discuss weighted equidistribution results for orbital measures (with no restriction to a homology class).
In Section \ref{section:therm-yoga}, we introduce a pressure function on cohomology,
which enables us to pick out a unique cohomology class associated to the weighting function
and hence a particular equilibrium state which gives the limit in Theorem \ref{weighted-equi-intro}.
In Section \ref{section:equi-null}, we use the approach of \cite{bab-led} to establish equidistribution of weighted null-homologous periodic orbits to this equilibrium state.
In the second part of the paper, starting with Section \ref{section:linking}, we apply our
equidistribution results to prove Theorem \ref{thm:hel-homfull-intro} and related results.
In Section \ref{section:linking}, we recall standard material about vector fields, forms, linking numbers
of knots, linking forms and helicity.
In Section \ref{section:per_orbits_helicity}, we relate weighted averages of linking numbers to 
integral involving a linking form and prove Theorem \ref{thm:hel-homfull-intro} (restated as Theorem \ref{thm:hel-homfull})
subject to a more general result, Theorem \ref{thm:main-homfull}.
In Section \ref{section:proof_of_main}, we obtain bounds on the linking form, which, combined
with the arguments of Contreras \cite{contreras}, enable us
to prove Theorem \ref{thm:main-homfull}.

\section{Anosov flows and homology}\label{section:anosov}
Let $M$ be a smooth closed Riemannian manifold (of arbitrary dimension) and
let $X^t:M\to M$ be an Anosov flow generated by the vector field $X$.
This means that the tangent bundle has a continuous $DX^t$-invariant splitting $TM = E^0 \oplus E^s \oplus E^u$,
where $E^0$ is the one-dimensional bundle spanned by $X$ and where there exist constants $C,\lambda >0$ such that
\begin{enumerate}
\item
$\|DX^tv\| \le Ce^{-\lambda t} \|v\|$, for all $v \in E^s$ and $t >0$;
\item
$\|DX^{-t}v\| \le Ce^{-\lambda t} \|v\|$, for all $v \in E^u$ and $t>0$.
\end{enumerate}
This class of flows was introduced by Anosov \cite{anosov}; for a good modern reference
see \cite{fisher}.
%In addition, the flows we consider will always be transitive, i.e. they have a dense orbit; this will %follow from the assumption of homological fullness that we introduce below.
In addition, we assume that $X^t : M \to M$ is topologically transitive, i.e. that there is a dense 
orbit. 

We say that the flow is topologically weak mixing if the equation 
$\psi \circ X^t = e^{iat} \psi$, for $\psi : M \to \mathbb C$ continuous and $a \in \mathbb R$, only
has the trivial solution where $\psi$ is constant and $a=0$.
A classical result of Plante \cite{plante} is that a transitive Anosov flows fails to be weak-mixing if and 
only if it is the constant time suspension of an Anosov diffeomorphism.
We also have that $X^t$ fails to be weak-mixing if and only if
$\{\ell(\gamma) \hbox{ : } \gamma \in \mathcal P\}$ is contained in a discrete subgroup of $\mathbb R$, where $\mathcal P$ is the set of prime periodic orbits defined in the introduction and 
$\ell(\gamma)$ is the least period of $\gamma \in \mathcal P$.

We wish to consider the real homology group $H_1(M,\mathbb R)\cong \mathbb R^b$, where
$b \ge 0$ is the first Betti number of $M$.
For $\gamma \in \mathcal P$, we write $[\gamma]$ for the corresponding real homology class
in $H_1(M,\mathbb R)$. (Note that $[\gamma]$ may be identified with the torsion free part of the
integral homology class of $\gamma$ in $H_1(M,\mathbb Z) \cong \mathbb Z^b \oplus \mathrm{Tor}$,
where $\mathrm{Tor}$ is a finite abelian group.)

We will need the following result.

\begin{proposition}[Parry and Pollicott \cite{PP86}]\label{prop:periodic_orbits_generate_homology}
Let $X^t : M \to M$ be a transitive Anosov flow. Then the set of integral homology classes of 
periodic orbits generates $H_1(M,\mathbb Z)$.
\end{proposition}

Let $\mathcal M(X)$ denote the space of $X^t$-invariant Borel probability measures on $M$
(with the weak$^*$ topology).
For each $\nu \in \mathcal M(X)$, there is
an associated homology class $\Phi_\nu \in H_1(M,\mathbb R)$,
called the {\it winding cycle} (or asymptotic cycle) for the measure. These cycles
were
introduced by Schwartzman 
\cite{Sch}. We define $\Phi_\nu$ using the duality $H_1(M,\mathbb R) =H^1(M,\mathbb R)^*$ 
and the formula
\[
\langle \Phi_\nu,[\omega]\rangle = \int \omega(X) \, d\nu,
\]
where $\omega$ is a closed $1$-form on $M$, $[\omega] \in H^1(M,\mathbb R)$ is its cohomology class,
 and
\[
\langle \cdot, \cdot \rangle : H_1(M,\mathbb R) \times H^1(M,\mathbb R) \to \mathbb R
\]
is the duality pairing. This is well defined since if $[\omega']=[\omega]$ then $\omega$ and $\omega'$
differ by an exact form $d\theta$, say, and we have
\[
\int d\theta(X) \, d\nu = \int L_X\theta \, d\nu =0,
\]
where $L_X$ is the Lie derivative,
\[
L_X\theta(x) = \lim_{t \to 0} \frac{1}{t} (\theta(X^t x) - \theta(x)).
\]
(That the final integral vanishes follows from the invariance of $\nu$ and the
dominated convergence theorem.)

We say that $X$ is {\it homologically full} if 
every integral homology class in $H_1(M,\mathbb Z)$ is represented by a periodic orbit.
The following proposition is a consequence of the results in \cite{Sh93}.

\begin{proposition}\label{prop:equiv_to_hom_full}
The following are equivalent:
\begin{enumerate} 
\item[(i)] $X$ is homologically full;
\item[(ii)]
the map $[\, \cdot \,] : \mathcal P \to H_1(M,\mathbb Z)/\mathrm{Tor}$ is a surjection;
\item[(iii)]
$0 \in \mathrm{int}(\{\Phi_\nu \hbox{ : } \nu \in \mathcal M(X)\})$.
\end{enumerate}
\end{proposition}

If $X^t$ is the suspension of a diffeomorphism then it cannot be homologically full
\cite{fried}. In particular, if a transitive Anosov flow is homologically full then it is 
automatically weak-mixing.

For each $\alpha \in H_1(M,\mathbb Z)/\mathrm{Tor}$, write
\[
\mathcal P(\alpha) = \{\gamma \in \mathcal P \hbox{ : } [\gamma]=\alpha\}.
\]

Our equidistribution results apply to Anosov flows in any dimension but when we consider linking 
of periodic orbits and helicity, we need to restrict to flows on $3$-manifolds.
In this case, we say that $M$ is a {\it real homology $3$-sphere} if it
has the same
real homology as $S^3$. 
This amounts to $M$ being a connected manifold
(so that $H_0(M,\mathbb R)\cong\mathbb R$) with $H_1(M,\mathbb R)=\{0\}$, since Poincar\'e duality gives
the homology in the remaining dimensions.
If $X^t : M\to M$ is a transitive Anosov flow on a real homology $3$-sphere then it is automatically 
homologically full (and hence weak-mixing) and $\mathcal P(0)=\mathcal P$.

\begin{example}\label{ex:orbifold}
There are a wealth of examples of transitive Anosov flows on real homology $3$-spheres. 
The simplest are provided by the following.
Let $S=\H^2/\Gamma$ be a hyperbolic $2$-orbifold of genus zero. By Theorem 13.3.6 in \cite{thurston}, this means that $\Gamma$ is a Fuchsian group acting freely and properly discontinuously on the hyperbolic plane $\mathbb H^2$, such that the number of cone points $p$ of $S$ satisfies $p\geq 5$, or $p=4$ and the orders are not all 2, or $p=3$ and the orders satisfy that the sum of their reciprocals is smaller than 1. We can think of $M=(T^1\H^2)/\Gamma\cong \PSL_2(\R)/\Gamma$ as the unit tangent bundle of $S$. Now, $M$ is a compact real homology 3-sphere (see Lemma 2.1 in \cite{dehornoy}), and the flow $X^t:M \to M$, given by the quotient of the geodesic flow on $\H^2$, is Anosov.
In this case, one should see the calculation of helicity in Example 2.2.1 of \cite{verjovsky-vila}.

Surgery techniques give rise to more examples. For example, consider the suspension flow of
the Anosov diffeomorphism on $\mathbb T^2$ induced by
the matrix $\left(\begin{smallmatrix} 2 & 1 \\ 1 & 1
\end{smallmatrix}\right)$. The complement of the flow orbit of the fixed point of the diffeomorphism at $0$ is homeomorphic to the complement of the figure eight knot in $S^3$. 
Goodman \cite{goodman} showed how to use Dehn surgeries on this knot complement to
produce new examples of Anosov flows on real homology $3$-spheres
(see also the more recent work of Foulon and Hasselblatt \cite{foulon-hasselblatt}).
\end{example}

\section{Pressure and equilibrium states}\label{section:pressure}
Let $X^t : M \to M$ be a transitive Anosov flow and let 
$\nu \in \mathcal M(X)$. We let $h(\nu)$ denote the measure-theoretic entropy of the time-$1$ map
$X^1 : M \to M$ with respect to $\nu$. Then the topological entropy $h(X)$ satisfies the variational 
principle
\[
h(X) = \sup\{h(\nu) \hbox{ : } \nu \in \mathcal M(X)\}.
\]
More generally, for a continuous function
$\varphi : M \to \mathbb R$,
we can define the {\it pressure} $P(\varphi)$ of $\varphi$ by
\begin{equation}
P(\varphi) = \sup\left\{h(\nu) + \int \varphi \, d\nu \hbox{ : } \nu \in \mathcal M(X)\right\}.
\end{equation}
If $\varphi$ is H\"older continuous then the supremum is attained at a unique measure 
$\mu_\varphi$, which we call the {\it equilibrium state} for $\varphi$ \cite{PP}.
It is immediate from the definition of pressure that if $\varphi \le \psi$ then
$P(\varphi) \le P(\psi)$.
The equilibrium state of a H\"older continuous function is ergodic and fully
supported. 

We will be particularly interested in the case where $\varphi = \varphi^u$,
defined in the introduction.
Then 
$\mu_{\varphi^u}$ is the Sinai--Ruelle--Bowen (SRB) measure for $X$, i.e., the
unique $X^t$-invariant probability measure which is absolutely
continuous with respect to the volume on $M$ \cite{bowen_ruelle}.
In particular, if $X^t$ preserves the Riemannian volume $m$ then
$m = \mu_{\varphi^u}$.

We now give an alternative, topological, definition of pressure, which we will use later. For further details of this definition, see Chapter 4 of \cite{fisher}.
Given $T,\delta>0$, we call the set 
$$
B(x,\delta,T):=\{y\in M:d(X^tx,X^ty)<\delta\text{ for all }0\leq t< T\}.
$$ 
a \textit{Bowen ball} around $x\in M$. Then, for $\delta>0$, a set $E\subset M$ is $(T,\delta)$\textit{-spanning} if 
$$
M=\bigcup_{x\in E}B(x,\delta,T).
$$ 
On the other hand, $E$ is $(T,\delta)$\textit{-separated} if whenever $x,y\in E$ with $x\neq y$, 
we have
$$
\max_{0\leq t <T} d(X^tx,X^ty) \geq\delta.
$$ 
We can then define pressure in the following way.
\begin{align*}
P(\varphi)&=\lim_{\delta\to 0}\limsup_{T\to\infty}\frac{1}{T}\log\inf\bigg\{\sum_{x \in E}
e^{\int_0^T \varphi(X^tx) \, dt}:E\,(T,\delta)\text{-spanning}\bigg\}\\
&=\lim_{\delta\to 0}\limsup_{T\to\infty}\frac{1}{T}\log\sup\bigg\{\sum_{x \in E}
e^{\int_0^T \varphi(X^tx) \, dt}:E\,(T,\delta)\text{-separated}\bigg\}.
\end{align*}
This definition is used in Section
\ref{section:proof_of_main}.

The characterisation of homologically full transitive Anosov flows in part (iii) of Proposition 
\ref{prop:equiv_to_hom_full} may be modified to give a statement in terms of the equilibrium states of H\"older continuous functions, as follows.

\begin{proposition}[Sharp \cite{Sh93}]\label{prop:hom_full_eq_state}
$X$ is homologically full if and only if there exists a H\"older continuous function 
$\varphi : M \to \mathbb R$ with $\Phi_{\mu_\varphi}=0$.
\end{proposition}

We will use the following lemma when we discuss
large deviations in Section \ref{section:equi-null}.

\begin{lemma} \label{entropyusc}
The map $\mathcal M(X) \to \mathbb R : \nu \mapsto h(\nu)$ is upper semi-continuous
and
\[
h(\nu) = \inf \left\{P(\varphi) - \int \varphi \, d\nu \hbox{ : } \varphi \in C(M,\mathbb R)\right\}.
\]
\end{lemma}

%The proof of Lemma \ref{entropyusc} is completely analogous to those of Theorem 8.2
%and Theorem 9.12
%of \cite{walters}, which deal with transformations rather than flows. In particular, the first 
%statement follows from the fact that the flow is expansive. 
The first statement in Lemma \ref{entropyusc} follows from the fact that the flow is expansive
(see Remark 4.3.18 and Corollary A.3.14 of \cite{fisher}).
Once we have established upper semi-continuity, rearranging the variational principle above into this form follows the same argument as
the proof of Theorem 9.12
of \cite{walters}.

We will need to use the notion of functions being cohomologous with respect to
$X$.
We say that continuous functions $\varphi,\psi : M \to \mathbb R$ are $X$-cohomologous
if there is a continuous function $u : M \to \mathbb R$ that is differentiable along flow lines
satisfying
\[
\varphi-\psi = L_Xu.
\]
Two $X$-cohomologous functions have the same integral with respect to every measure in $\mathcal M(X)$. For a constant $c \in \mathbb R$, we have
\[
P(\varphi +L_Xu+c) = P(\varphi)+c.
\]

For $\gamma \in \mathcal P$, let
\[
\int_\gamma \varphi := \int_0^{\ell(\gamma)} \varphi(X^t x_\gamma)
\, dt
\]
with
$x_\gamma \in \gamma$. If $\varphi$ and $\psi$ are $X$-cohomologous then it is clear that,
for every $\gamma \in \mathcal P$,
$\int_\gamma \varphi = \int_\gamma \psi$. However, we also have the following converse.

\begin{lemma}[Livsic \cite{livsic}]\label{lem:livsic}
Suppose that $\varphi,\psi : M \to \mathrm R$ are H\"older continuous. If
\[
\int_\gamma \varphi = \int_\gamma \psi
\quad \forall \gamma \in \mathcal P
\]
then $\varphi$ and $\psi$ are $X$-cohomologous.
\end{lemma}

We will use the following result later.

\begin{lemma}\label{lem:coho-to-neg}
Suppose that $\varphi : M \to \mathbb R$ is H\"older continuous. 
Then there exists $\epsilon>0$ and a H\"older
continuous function $v : M \to \mathbb R$ such that, for all $x \in M$ and $T \ge 0$, we have
\[
\int_0^T \varphi(X^tx) \, dt  \le (P(\varphi)-\epsilon) T + v(X^Tx) -v(x).
\]

\end{lemma}

\begin{proof}
Since $P(\varphi -P(\varphi))=P(\varphi)-P(\varphi)=0$,
without loss of generality, we may assume that $P(\varphi)=0$.
We use the standard result that $X^t : M \to M$ may be modelled by a suspended flow over a mixing 
subshift of finite type $\sigma : \Sigma \to \Sigma$, with a strictly positive H\"older continuous roof function $r : \Sigma \to \mathbb R$ \cite{bowen2}. More precisely, let 
\[
\Sigma^r =
\{(x,\tau) \in \Sigma \times \mathbb R \hbox{ : } 0 \le \tau \le r(x)\}/\sim,
\] 
where we have the identifications
$(x,r(x)) \sim (\sigma x,0)$, and let $\sigma^t : \Sigma^r \to \Sigma^r$ be the flow
$\sigma^t(x,\tau) = (x,\tau+t)$ modulo $\sim$.
Then there is a H\"older continuous surjection $\pi : \Sigma^r \to M$ that semi-conjugates
$\sigma^t$ and $X^t$ and is sufficiently close to 
being a bijection that the pressure of a function with respect to $X^t$ and of its pull-back by $\pi$
are equal. The pressure of a function $q : \Sigma \to \mathbb R$
with respect to $\sigma$ may be defined to be
\[
P_\sigma(q) = \sup\left\{h_\sigma(m) + \int q \, dm \hbox{ : } m \in \mathcal M(\sigma)\right\},
\]
where $\mathcal M(\sigma)$ is the set of $\sigma$-invariant probability measures on $\Sigma$
and $h_\sigma(m)$ is the measure-theoretic entropy.
If we define a function $q_\varphi : \Sigma \to \mathbb R$ by
\[
q_\varphi(x) = \int_0^{r(x)} \varphi(\pi(x,\tau)) \, d\tau
\]
then the relationship between pressure with respect to $\sigma$ and with respect to the 
suspended flow gives that 
\[
P_\sigma(q_\varphi)= P_\sigma(-P(\varphi) r+q_\varphi)=0.
\]
It then follows that $q_\varphi$ is $\sigma$-cohomologous to a strictly negative function, i.e. that there exists a continuous function $u : \Sigma \to \mathbb R$ such that
$q_\varphi + u\circ \sigma -u$ is strictly negative and, in fact,
bounded above by 
$-\epsilon \|r\|_\infty$,
for some $\epsilon>0$
(see Chapter 7 of \cite{PP}).
We then have that
\[
 \int_\gamma (\varphi +\epsilon)
\le 0,
\]
for all $\gamma \in \mathcal P$.
It then follows from Theorem 1 of \cite{PS04} that
there exists a H\"older continuous function $v : M \to \mathbb R$ 
such that, for all $x \in M$ and $T \ge 0$,
\[
\int_0^T \varphi(X^tx) \, dt + \epsilon T \le v(X^Tx) -v(x).
\]
\end{proof}

We also need to consider functions of the form
\[
\mathbb R^b \to \mathbb R^b : (t_1,\ldots,t_b) \mapsto P(\varphi + t_1\psi_1 +
\cdots t_b \psi_b),
\]
for $b \ge 1$ and H\"older continuous functions $\varphi,\psi_1,\ldots,\psi_b : M \to \mathbb R$.
For subshifts of finite type results on differentiaing such functions are standard
(see, for example, \cite{PP}). The calculations for hyperbolic flows are carried out in \cite{lalley} and \cite{Sh92}.

\begin{lemma}[\cite{lalley},\cite{Sh92}]\label{convexity_of_pressure}
Let
$\varphi,\psi : M \to \mathbb R$
be H\"older continuous functions. Then the function
\[
\mathbb R \to \mathbb R : (t_1,\ldots,t_b) \mapsto P(\varphi+t_1\psi_1 + \cdots t_b\psi_b).
\]
is convex and real-analytic with
\begin{equation*}
\frac{\partial P(\varphi+t_1\psi_1 + \cdots + t_b\psi_b)}{\partial t_i}\bigg|_{(t_1,\ldots,t_b)=0} = \int  \psi_i \, d\mu_\varphi.
\end{equation*}
Furthermore, unless $a_1\psi_1 + \cdots a_b\psi_b$ is $X$-cohomologous to a constant for some
$(a_1,\ldots,a_b) \ne 0$,
$t \mapsto P(\varphi+t\psi)$ is strictly convex and
\begin{equation*}
\det \nabla^2P(\varphi+t_1\psi_1 + \cdots + t_b\psi_b)|_{(t_1,\ldots,t_b)=0}
>0.
\end{equation*}
\end{lemma}

\section{Equidistribution of periodic orbits}\label{section:equi_per_orbits}

In this section, we discuss weighted equidistribution results for periodic orbits
(with no restriction on the homology class).
For a continuous function $\varphi : M \to \mathbb R$ and $a<b$, write
\[
\pi_\varphi(T,\mathbbm{1}_{[a,b]})
=
\sum_{\gamma \in \mathcal P} 
\mathbbm{1}_{[a,b]}(\ell(\gamma)-T)
e^{\int_\gamma \varphi}.
\]
For $\gamma \in \mathcal P$, let $\mu_\gamma$ be the probability measure defined by
\[
\int \psi \, d\mu_\gamma= \frac{1}{\ell(\gamma)} \int_\gamma \psi.
\]
We will discuss the proof of the following equidistribution result.

\begin{theorem} \label{weighted-nonnull-equi}
Let $X^t : M \to M$ be a weak-mixing Anosov flow. 
Let $\varphi : M \to \mathbb R$ be H\"older continuous.
Then, for $a<b$, the measures
\[
\frac{1}{\pi_\varphi(T,\mathbbm{1}_{[a,b]})}
\sum_{\gamma \in \mathcal P} 
\mathbbm{1}_{[a,b]}(\ell(\gamma)-T)
e^{\int_\gamma \varphi} \mu_\gamma
\]
converge weak$^*$ to $\mu_{\varphi}$, as $T \to \infty$,
and the same holds if we replace $[a,b]$ by $(a,b)$, $(a,b]$ or $[a,b)$.
\end{theorem}

The case where $\varphi=0$ is a classical theorem of Bowen \cite{bowen-equi} and was 
reproved using zeta function techniques by Parry \cite{parry-equi}.
For $\varphi = \varphi^u$, the result was proved by Parry \cite{parry-cmp} and the same arguments cover the cases where $P(\varphi) \ge 0$ (see \cite{parry-maryland}, \cite{PP}).
Roughly speaking, if $P(\varphi)>0$ then the analytic properties of a dynamical zeta function can be used to show that, for a strictly positive H\"older continuous function $\psi : M \to \mathbb R$, we have
\begin{equation}\label{parry-weighted-asymptotic}
\sum_{\ell(\gamma) \le T} \left(\int_\gamma \psi\right)  e^{\int_\gamma \varphi} 
\sim \left(\int \psi \, d\mu_\varphi\right) \frac{e^{P(\varphi)T}}{P(\varphi)},
\end{equation}
and then the required result is obtained via elementary arguments.
This can be extended to $P(\psi)=0$ by approximation.
However, if $P(\varphi)<0$ then a different argument is needed.
This is based on the large deviations results of Kifer \cite{kifer} and an estimate on the
growth of $\pi_\varphi(T,\mathbbm{1}_{[a,b]})$. 
A sufficient estimate is claimed as Proposition 3 of
\cite{Po1995}, where it is is attributed to Parry \cite{parry-maryland}, but the quoted result was only
stated and proved by Parry
when $P(\varphi)>0$.
To fill this gap, we show the following.

\begin{lemma}\label{lem:growth_of_weighted_per_orbits}
For any continuous function $\varphi : M \to \mathbb R$ and $a<b$, we have
\[
\lim_{T \to \infty} \frac{1}{T} \log \pi_\varphi(T,\mathbbm{1}_{[a,b]}) = P(\varphi).
\]
\end{lemma}

\begin{proof}
We start by assuming that $\varphi$ is H\"older continuous and that $P(\varphi)>0$. The zeta function
\[
\zeta(s,\varphi) = \prod_{\gamma \in \mathcal P} \left(1-e^{-s\ell(\gamma) + \int_\gamma
\varphi}\right)^{-1},
\]
which converges for $\mathrm{Re}(s)>P(\varphi)$, has a non-zero analytic
extension to $\mathrm{Re}(s) \ge P(\varphi)$, apart from a simple pole
at $s=P(\varphi)$ \cite{PP}. The same is true
for the function
\[
\sum_{\gamma \in \mathcal P} e^{-s\ell(\gamma)+\int_\gamma \varphi}
\]
(the proof is an application of Lemma \ref{same_aoc} below).
We can then deduce that
\[
 \sum_{\ell(\gamma) \le T} e^{\int_\gamma \varphi}
\sim \frac{e^{P(\varphi)T}}{P(\varphi)T},
\]
as $T \to \infty$, and hence that
\[
\lim_{T \to \infty} \frac{1}{T} \log \pi_\varphi(T,\mathbbm{1}_{[a,b]]}) =P(\varphi).
\]
For an arbitrary H\"older continuous $\varphi$, choose $c>0$ such that $P(\varphi)+c>0$
and note that
\[
e^{-c(T+b)}\pi_{\varphi+c}(T,\mathbbm{1}_{[a,b]})
\le
\pi_\varphi(T,\mathbbm{1}_{[a,b]})
\le
e^{-c(T+a)} \pi_{\varphi+c}(T,\mathbbm{1}_{[a,b]}).
\]
This gives us
\[
\lim_{T \to \infty} \frac{1}{T} \log \pi_\varphi(T,\mathbbm{1}_{[a,b]]}) =-c +P(\varphi+c)
= P(\varphi).
\]
Finally, if $\varphi$ is only continuous, given $\epsilon>0$, we can find a H\"older continuous function $\varphi'$ with $\|\varphi-\varphi'\|_\infty<\epsilon$, so that
\begin{align*}
P(\varphi) -2\epsilon \le P(\varphi')-\epsilon
&\le
\liminf_{T \to \infty}
\frac{1}{T} \log \pi_\varphi(T,\mathbbm{1}_{[a,b]]})
\\
&\le
\limsup_{T \to \infty}
\frac{1}{T} \log \pi_\varphi(T,\mathbbm{1}_{[a,b]]})
\le P(\varphi')+\epsilon \le P(\varphi)+2\epsilon,
\end{align*}
which gives the required limit.
\end{proof}

\begin{remark}
One can improve the growth rate estimate to
\[
\pi_\varphi(T,\mathbbm{1}_{[a,b]}) \sim \left(\int_a^b e^{P(\varphi)x} \, dx\right) \frac{e^{P(\varphi)T}}{T},
\]
as $T \to \infty$,
using a simplified version of the proof of Theorem \ref{weighted-alpha-asymptotic} below.
\end{remark}

For a compact set $\mathcal K\subset \mathcal M(X)$, write
\[
\Xi_\varphi(T,\mathbbm{1}_{[a,b]},\mathcal K)
= \sum_{\substack{\gamma \in \mathcal P \\ \mu_\gamma \in \mathcal K}}
\mathbbm{1}_{[a,b]}(\ell(\gamma)-T)
\exp\left( \int_\gamma \varphi \right).
\]
The growth rate result in Lemma \ref{lem:growth_of_weighted_per_orbits} implies the following large deviations estimate from \cite{Po1995}, where it appears as Theorem 1.

\begin{theorem}[Pollicott \cite{Po1995}]\label{th:ld}
Let $X^t : M \to M$ be a weak-mixing Anosov flow and let $\varphi : M \to \mathbb R$ 
be a H\"older
continuous function.
Then, for every compact set $\mathcal K \subset \mathcal M(X)$ such that $\mu_{\varphi}
\notin \mathcal K$ and $a>b$, we have
\[
\limsup_{T \to \infty} \frac{1}{T} \log 
\left(\frac{\Xi_\varphi(T,\mathbbm{1}_{[a,b]},\mathcal K)}{\pi_\varphi(T,\mathbbm{1}_{[a,b]})}\right)<0.
\]
\end{theorem}

Theorem \ref{weighted-nonnull-equi} then follows from this. We omit the proof as it is almost identical to the proof that Theorem \ref{th:ld} implies Theorem \ref{weighted-equi} below.

\begin{remark}
Theorem \ref{weighted-nonnull-equi} also holds more generally for hyperbolic flows (with the same proof).
If $X^t$ is not weak-mixing then the periods $\ell(\gamma)$ are all integer multiples of some
$c>0$ and the result holds in this case provided $b-a\ge c$.
\end{remark}

\section{Pressure and cohomology}\label{section:therm-yoga}

Let $X^t : M \to M$ be a homologically full transitive Anosov flow on a manifold $M$ whose first Betti number $b = \dim H_1(M,\mathbb R)$ is at least $1$. Let 
$\varphi : M \to \mathbb R$ be a H\"older continuous function.
In this section, we define a pressure function on the cohomology group 
$H^1(M,\mathbb R)$ that will allow us to identify the growth rate of periodic orbits in a fixed homology 
class weighted by $\varphi$.
(This generalises results in \cite{Sh93} which were restricted to the case $\varphi=0$.)

For a closed $1$-form $\omega$ on $M$, let $f_\omega : M \to \mathbb R$ denote the function
$f_\omega = \omega(X)$, i.e. for $x \in M$, $f_\omega(x) = \omega(X(x))$. If $\omega'$ is in the same cohomology class then, as above,
$\omega - \omega' =d\theta$ and, for any periodic orbit $\gamma$,
\begin{align*}
\int_\gamma f_\omega -\int_\gamma f_{\omega'} 
&= \int_\gamma d\theta(X) = \int_\gamma L_X \theta =0,
\end{align*}
so, by Lemma \ref{lem:livsic}, $f_\omega$ and $f_{\omega'}$ are $X$-cohomologous. Thus, 
we can define
$f_{[\omega]}$, where $[\omega]$ is the cohomology class of $\omega$,
as a function on $M$ up to $X$-cohomology. In particular, this is sufficient for us to use 
$f_{[\omega]}$ to define a pressure function.

Define $\beta_\varphi : H^1(M,\mathbb R) \to \mathbb R$ by
\[
\beta_\varphi([\omega]) = P(\varphi + f_{[\omega]}).
\]
The next result identifies the growth rate we require and the minimum of $\beta_\varphi$.

\begin{proposition}\label{strictly_convex_and_finite_min} Let $\varphi : M \to \mathbb R$ be H\"older continuous.
Then $\beta_\varphi$ is strictly convex and
 there exists a unique $\xi(\varphi) \in H^1(M,\mathbb R)$
such that
\[
\beta_\varphi(\xi(\varphi)) = \inf_{[\omega] \in H^1(M,\mathbb R)} \beta_\varphi([\omega]).
\]
Furthermore, $\mu_{\varphi +f_{\xi(\varphi)}}$ is the unique probability measure satisfying
\[
h(\mu_{\varphi +f_{\xi(\varphi)}}) + \int \varphi \, d\mu_{\varphi +f_{\xi(\varphi)}}
= \sup \left\{h(\nu) + \int \varphi \, d\nu \hbox{ : } \nu \in \mathcal M(X), \ \Phi_\nu=0\right\}.
\]
\end{proposition}

\begin{proof}
Fix a basis $c_1,\ldots,c_b$ for the free $\mathbb Z$-module $H_1(M,\mathbb Z)/\mathrm{Tor}$
(regarded as a lattice in $H_1(M,\mathbb R)$) and 
take the dual basis 
$w_1,\ldots,w_b$
for $H^1(M,\mathbb R)$, i.e.
\[
\langle c_i ,w_j \rangle = \delta_{ij},
\]
where $\delta_{ij}$ is the Kronecker symbol.

Write 
$f_i = f_{w_i}$, $i=1,\ldots,b$. Then we can regard $\beta_\varphi$ as a function
$\beta_\varphi : \mathbb R^b \to \mathbb R$ given by
\[
\beta_\varphi(t) = P(\varphi +t_1f_1 + \cdots +t_bf_b).
\]
By Lemma \ref{convexity_of_pressure}, this function 
is strictly convex unless there is a non-zero $a=(a_1,\ldots,a_b) \in \mathbb R^b$ such that
$a_1f_1 + \cdots a_bf_b$ is $X$-cohomologous to a constant. Suppose there is such a
non-zero vector $a$. Since the flow is homologically 
full, we can find $\mu \in \mathcal M(X)$ with $\Phi_\mu=0$, which is equivalent to
$\int f_i \, d\mu=0$, $i=1,\ldots,b$, and therefore $a_1f_1 + \cdots +a_bf_b$ is $X$-cohomologous to zero. In particular, for a prime periodic orbit $\gamma$, $$a_1\int_\gamma f_1 + \cdots +a_b\int_\gamma f_b=0.$$ Also, by definition of $f_i$, the real homology class of $\gamma$  is given by $$\left(\int_\gamma f_1,\ldots, \int_\gamma f_b\right).$$  Thus the homology of each periodic orbit 
is constrained to lie in the hyperplane
$a_1x_1 + \cdots +a_bx_b=0$ in $H_1(M,\mathbb R) \cong \mathbb R^b$. This contradicts
Proposition \ref{prop:periodic_orbits_generate_homology},
that the homology classes of periodic orbits generate $H_1(M,\mathbb Z)$ as a group.
So $\beta_\varphi$ is strictly convex.

We now show that $\beta_\varphi$ has a finite minimum. Since $\beta_\varphi$ is strictly convex,
this minimum will automatically be unique. Since the flow is homologically full, it follows from \cite{Sh93} that
$\beta_0$ is strictly convex and has a finite minimum. Noting that
\[
|\beta_\varphi(t) - \beta_0(t)| \le \|\varphi\|_\infty,
\]
we see that $\beta_\varphi$ also has a finite minimum. 
We call the point where the minimum occurs $\xi(\varphi)$. 
Clearly, $\nabla \beta_\varphi(\xi(\varphi))=0$.

 Writing $f=(f_1,\ldots,f_b)$, we have
\[
\nabla \beta_\varphi(t) = \int f \, d\mu_{\varphi + t_1f_1 + \cdots + t_bf_b},
\]
so
\[
\int f \, d\mu_{\varphi + f_{\xi(\varphi)}} = \nabla \beta_\varphi(\xi(\varphi))=0.
\]
In terms of the winding cycle, this gives $\Phi_{\mu_{\varphi + f_{\xi(\varphi)}}} =0$.
Suppose $\nu$ is an arbitrary measure in $\mathcal M(X)$ with $\Phi_\nu=0$. Then $\int f_{\xi(\varphi)} \, d\nu=0$ and, using the definition of
equilibrium state,
\begin{align*}
h(\nu) + \int \varphi \, d\nu &= h(\nu) + \int (\varphi + f_{\xi(\varphi)}) \, d\nu \\
&\le h(\mu_{\varphi +f_{\xi(\varphi)}}) + \int (\varphi + f_{\xi(\varphi)})\, d\mu_{\varphi +f_{\xi(\varphi)}}
\\
&= h(\mu_{\varphi +f_{\xi(\varphi)}}) + \int \varphi \, d\mu_{\varphi +f_{\xi(\varphi)}}.
\end{align*}
Thus,
\[
h(\mu_{\varphi +f_{\xi(\varphi)}}) + \int \varphi \, d\mu_{\varphi +f_{\xi(\varphi)}}
= \sup \left\{h(\nu) + \int \varphi \, d\nu \hbox{ : } \nu \in \mathcal M(X), \ \Phi_\nu=0\right\},
\]
as required.
\end{proof}

Finally, we note that it is part of the standard theory of the pressure function for hyperbolic flows that
the exponential of the function $t \mapsto \beta_\varphi(t)$ has an analytic extension
\[
\mathcal D \to \mathbb C : s \mapsto \exp \beta_\varphi(s),
\]
where $\mathcal D$ is a neighbourhood of $\mathbb R^b$ in $\mathbb C^b$
(see, for example, the discussions in Section 1 of \cite{KS} or Section 2 of \cite{Sh92}).

\section{Equidistribution of null-homologous orbits}\label{section:equi-null}

\subsection{Weighted asymptotics for orbits in a fixed homology class}
Let $\varphi : M \to \mathbb R$ be H\"older continuous and let
$\beta_\varphi$ and $\xi(\varphi)$ be defined as in Section \ref{section:therm-yoga}.
To lighten the notation, we shall write
\[
\xi = \xi(\varphi) \quad \mbox{and} \quad \beta = \beta_\varphi(\xi(\varphi)).
\]

For $\alpha \in H_1(M,\mathbb Z)/\mathrm{Tor}$,
a compactly supported function $g : \mathbb R \to \mathbb R$
and $T>0$, write
\[
\pi_\varphi(T,\alpha,g) = \sum_{\gamma \in \mathcal P(\alpha)} g(\ell(\gamma)-T)
\exp\left( \int_\gamma \varphi \right).
\]
We have the following asymptotic formula.

\begin{theorem}\label{weighted-alpha-asymptotic}
Let $X^t : M \to M$ be a homologically full transitive Anosov flow and let $\varphi : M \to \mathbb R$ 
be a H\"older
continuous function. Then, for every compactly supported continuous function
$g : \mathbb R \to \mathbb R$, we have
\[
\pi_\varphi(T,\alpha,g) \sim \frac{1}{(2\pi)^{b/2}\sqrt{\det \nabla^2 \beta_\varphi(\xi)}} 
\left(\int_{-\infty}^\infty e^{\beta x} g(x) \, dx\right)
e^{-\langle \alpha,\xi \rangle}
\frac{e^{\beta T}}{T^{1+b/2}},
\]
as $T \to \infty$.
In particular,
\[
\lim_{T \to \infty} \frac{1}{T} \log \pi_\varphi(T,\alpha,g) = \beta.
\]
\end{theorem}

Applying a simple approximation argument, we immediately obtain the following
corollary. We shall see later in the section that this implies the equidistribution result we seek.

\begin{corollary}\label{cor:growth_of_weighted_null_hom}
\[
\lim_{T \to \infty} \frac{1}{T} \log \pi_\varphi(T,\alpha,\mathbbm{1}_{[a,b]}) = \beta
\]
and the same holds if we replace $[a,b]$ by $(a,b)$, $(a,b]$ or $[a,b)$.
\end{corollary}

We proceed with the proof of Theorem \ref{weighted-alpha-asymptotic},
following the analysis of \cite{bab-led}.
For $p \in \mathbb R$, $\delta_p$ denotes the Dirac measure giving mass $1$ to $p$.
For $\varsigma \in \mathbb R$,
define measures $\mathfrak M_{T,\alpha,\varphi,\varsigma}$ on $\mathbb R$ by
\[
\mathfrak{M}_{T,\alpha,\varphi,\varsigma}
= \sum_{\gamma \in \mathcal P(\alpha)} e^{-\varsigma \ell(\gamma)+ \int_\gamma \varphi +\langle \alpha,\xi\rangle}
\delta_{\ell(\gamma)-T}.
\]
Write $g_\varsigma(x) = e^{-\varsigma x}g(x)$.
We then have the following.

\begin{lemma}\label{lem:relation_between_pi_and_M}
For all $\varsigma \in \mathbb R$, we have
\[
\pi_\varphi(T,\alpha,g_\varsigma) = e^{\varsigma T - \langle \alpha,\xi \rangle}
\int g \, d\mathfrak M_{T,\alpha,\varphi,\varsigma}.
\]
\end{lemma}

\begin{proof}
The result follows from the direct calculation
\begin{align*}
\int g \, d\mathfrak M_{T,\alpha,\varphi,\varsigma}
&=
\sum_{\gamma \in \mathcal P(\alpha)}
g(\ell(\gamma)-T) e^{-\varsigma \ell(\gamma) +\int_\gamma \varphi + \langle \alpha,\xi \rangle}
\\
&= e^{-\varsigma T} 
\sum_{\gamma \in \mathcal P(\alpha)}
g(\ell(\gamma)-T) e^{-\varsigma (\ell(\gamma)-T) +\int_\gamma \varphi + \langle \alpha,\xi \rangle}
\\
&= e^{-\varsigma T}
g_\varsigma(\ell(\gamma)-T) e^{\int_\gamma \varphi +\langle \alpha ,\xi \rangle}
\\
&= e^{-\varsigma T + \langle \alpha,\xi \rangle} \pi_\varphi(T,\alpha,g_\varsigma). 
\end{align*}
\end{proof}

Now we introduce a complex function
\[
Z(s,w,z) 
= \sum_{\gamma \in \mathcal P} 
e^{-s \ell(\gamma) + w\int_\gamma \varphi + \langle [\gamma],z)\rangle},
\]
defined, where the series converges, for 
$(s,w,z) \in \mathbb C \times \mathbb C \times \mathbb C^b/2\pi i\mathbb Z^b$.
In fact, the $w$ only appears for book-keeping reasons and, ultimately, we shall set 
$w=1$. Furthermore, we are only interested in $z$ of the form
$z = \xi+iv$, with $v \in \mathbb R^b/2\pi \mathbb Z^b$.
We will relate $Z(s,w,z)$ to the logarithm of the zeta function
\[
\zeta(s,w,z) = \prod_{\gamma \in \mathcal P} 
\left(
1-e^{-s\ell(\gamma) + w\int_\gamma \varphi + \langle [\gamma],z)\rangle}
\right)^{-1}.
\]
We see that we have
\[
\log \zeta(s,w,z) = \sum_{n=1}^\infty \frac{1}{n} Z(ns,nw,nz).
\]

The standard theory of dynamical zeta functions \cite{PP} tells us that
$\zeta(s,1,\xi+iv)$ converges absolutely for
$\mathrm{Re}(s) > \beta$
and its analytic extension is well understood.
We need to show that $Z(s,1,\xi+iv)$ behaves like $\log \zeta(s,1,\xi+iv)$
and we do this by showing that their difference converges absolutely in a larger 
half-plane.

\begin{lemma}\label{same_aoc}
There exists $\epsilon>0$ such that $Z(s,1,\xi+iv) - \log \zeta(s,1,\xi+iv)$ converges
absolutely for $\mathrm{Re}(s) > \beta-\epsilon$.
\end{lemma}

\begin{proof}
As in the proof of Lemma \ref{lem:coho-to-neg} (to which we refer for notation), we use the fact that
there is 
a suspended flow $\sigma^t : \Sigma^r \to \Sigma^r$
over a mixing 
subshift of finite type $\sigma : \Sigma \to \Sigma$, with a strictly positive H\"older continuous roof function $r : \Sigma \to \mathbb R$. \cite{bowen2}, 
and a H\"older continuous surjection $\pi : \Sigma^r \to M$ that semi-conjugates
$\sigma^t$ and $X^t$ and is sufficiently close to 
being a bijection that the pressure of a function with respect to $X^t$ of of its pull-back by $\pi$
are equal. 
If we define a function $q_{\varphi +f_\xi} : \Sigma \to \mathbb R$ by
\[
q_{\varphi+f_\xi}(x) = \int_0^{r(x)} (\varphi + f_{\xi})(\pi(x,\tau)) \, d\tau
\]
then the relationship between pressure with respect to $\sigma$ and with respect to the 
suspended flow gives that 
\[
P_\sigma(-\beta r+q_{\varphi+f_\xi})=0.
\]
It then follows that $-\beta r+q_{\varphi+f_\xi}$ is $\sigma$-cohomologous to a strictly negative function, i.e. that there exists a continuous function $u : \Sigma \to \mathbb R$ such that
$-\beta r+q_{\varphi +f_\xi} + u\circ \sigma -u$ is strictly negative and, in fact,
bounded above by 
$-3 \epsilon\|r\|_\infty$, for some $\epsilon>0$
(see Chapter 7 of \cite{PP}).
We then have that
\[
-\beta  \ell(\gamma) 
+ \int_\gamma \varphi + \langle [\gamma],\xi\rangle 
\le -3\epsilon  \ell(\gamma),
\]
for all $\gamma \in \mathcal P$.

For $\varsigma > \beta -\epsilon$, we have
\[
-\varsigma \ell(\gamma) + \int_\gamma \varphi + \langle [\gamma],\xi\rangle 
\le -2\epsilon \ell(\gamma)
<0
\]
and hence
\begin{align*}
&|\log \zeta(\varsigma,1,\xi) - Z(\varsigma,1,\xi)|
=\sum_{n=2}^\infty \frac{1}{n} Z(n\varsigma,n,n\xi) 
\\
&\le \sum_{n=2}^\infty  Z(n\varsigma,n,n\xi)
= \sum_{n=2}^\infty \sum_{\gamma \in \mathcal P} 
e^{n(-\varsigma \ell(\gamma) + \int_\gamma \varphi + \langle [\gamma],\xi \rangle)}
\\
&=
\sum_{\gamma \in \mathcal P} 
\frac{e^{2(-\varsigma \ell(\gamma) + \int_\gamma \varphi + \langle [\gamma],\xi \rangle)}}
{1-e^{-\varsigma  \ell(\gamma) + \int_\gamma \varphi + \langle [\gamma],\xi )\rangle}}
\le 
C
\sum_{\gamma \in \mathcal P}
e^{-\varsigma \ell(\gamma) + \int_\gamma \varphi + \langle [\gamma],\xi \rangle}
e^{-2\epsilon  \ell(\gamma)}
\\
&= 
C
\sum_{\gamma \in \mathcal P} 
e^{-(\varsigma+2\epsilon ) \ell(\gamma) + \int_\gamma \varphi + \langle [\gamma],\xi \rangle}
=
C
\sum_{\gamma \in \mathcal P} 
e^{-(\beta+\epsilon ) \ell(\gamma) + \int_\gamma \varphi + \langle [\gamma],\xi \rangle}
\\
&= C Z(\beta+\epsilon,1,\xi) < \infty,
\end{align*}
where $C$ is some positive constant (depending on $\epsilon$).
\end{proof}

For $k \in \mathbb N$, let $C^k(\mathbb R \times \mathbb R^b/(2\pi \mathbb Z)^b,\mathbb R)$ denote the set of $C^k$ functions from $\mathbb R \times \mathbb R^b/(2\pi \mathbb Z)^b$ to $\mathbb R$, equipped with the topology of uniform convergence of the $j$th derivatives,
for $0 \le j \le k$, on compact sets.
The following result is, apart from the weighting by $\varphi$, a simplified
version of Proposition 2.1 in \cite{bab-led}.

\begin{proposition}
For each $k \in \mathbb{N}$, there exists
\begin{enumerate}
\item[(i)]
an open neighbourhood $U = U_1 \times U_2$ of $(0,0) \in \mathbb R \times \mathbb R^b/2\pi \mathbb Z^b$;
\item[(ii)]
a function $\rho \in C^k(\mathbb R \times \mathbb R^b/2\pi \mathbb Z^b,\mathbb R)$ which satisfies
$\rho(0,0)=1$ and vanishes outside of $U$;
\item[(iii)]
a function $A \in C^k(\mathbb R \times \mathbb R^b/2\pi \mathbb Z^b,\mathbb R)$ such that
\[
\lim_{\varsigma \downarrow \beta_\varphi(\xi(\varphi))} Z(\varsigma +it,1,\xi+iv)
= -\rho(t,v) \log(\beta +it - \beta_\varphi(\xi+iv)) + A(t,v),
\]
where $\beta_\varphi(u+iv)$ is an analytic extension of $\beta_\varphi(u)$ to 
$\{u \in \mathbb R^b \hbox{ : } \|u-\xi(\varphi)\|<\delta\} \times U_2$, for some small $\delta>0$. 
\end{enumerate}
In particular,
the function 
\[
\Upsilon(t,v) := \lim_{\varsigma \downarrow \beta_\varphi(\xi(\varphi))} Z(\varsigma+it,1, \xi+iv)
\]
is locally integrable on $\mathbb R \times \mathbb R^b/2\pi \mathbb Z^b$.
Furthermore, for any compact
$K \subset \mathbb R \times \mathbb R^b/2\pi \mathbb Z^b$, there exist constants
$C_1,C_2>0$ such that, for any $\varsigma > \beta$, we have
\[
|Z(\varsigma+it,\xi+iv)| \le 
\begin{cases}
-C_1 \log|\beta +it -\beta_\varphi(\xi(\varphi)+iv)| \mbox{ if } (t,v) \in U, \\
C_2 \mbox{ if } (t,v) \in K\setminus U.
\end{cases}
\]
\end{proposition}

\medskip
One can then proceed as in section 2 of \cite{bab-led} 
to show that, if we define
\[
\mathfrak{m}_T :=(2\pi)^{b/2} \sqrt{\det \nabla^2\beta_\varphi(\xi(\varphi))} T^{1+b/2}
e^{\langle \alpha,\xi \rangle}
\mathfrak{M}_{T,\alpha,\varphi,\beta},
\]
then
\[
\lim_{T \to \infty} \int_{\mathbb{R}} g \, d\mathfrak{m}_T = \int_{\mathbb{R}} g \, d\mathrm{Leb},
\]
for all continuous compactly supported $g : \mathbb{R} \to \mathbb{R}$. Finally, taking a continuous
compactly supported $g : \mathbb R \to \mathbb R$ and applying this to $g_\beta$, we
can use Lemma \ref{lem:relation_between_pi_and_M} to
obtain Theorem \ref{weighted-alpha-asymptotic}.

\subsection{Large deviations and weighted equidistribution}
For $\delta>0$, write
\[
\Xi_\varphi(T,\alpha,\mathbbm{1}_{[a,b]},\mathcal K)
= \sum_{\substack{\gamma \in \mathcal P(\alpha) \\ \mu_\gamma \in \mathcal K}}
\mathbbm{1}_{[a,b]}(\ell(\gamma)-T)
\exp\left( \int_\gamma \varphi \right).
\]
The growth rate result in Corollary \ref{cor:growth_of_weighted_null_hom} implies the following large deviations estimate.

\begin{theorem}\label{th:ld}
Let $X^t : M \to M$ be a homologically full transitive Anosov flow and let $\varphi : M \to \mathbb R$ 
be a H\"older
continuous function.
Then, for every compact set $\mathcal K \subset \mathcal M(X)$ such that $\mu_{\varphi + f_\xi}
\notin \mathcal K$ and $a<b$, we have
\[
\limsup_{T \to \infty} \frac{1}{T} \log 
\left(\frac{\Xi_\varphi(T,\alpha,\mathbbm{1}_{[a,b]},\mathcal K)}{\pi_\varphi(T,\alpha,\mathbbm{1}_{[a,b]})}\right)<0.
\]
The same holds if we replace $[a,b]$ by $(a,b)$, $(a,b]$ or $[a,b)$.
\end{theorem}

\begin{proof}
This is a standard type of argument which originates from the work of Kifer
(for example \cite{kifer-tams}).
Define a function $Q : C(M,\mathbb R) \to \mathbb R$ by
\[
Q(\psi) = P(\varphi + f_\xi + \psi).
\]
From Corollary \ref{cor:growth_of_weighted_null_hom}, we have
\begin{equation}\label{eq:ld_bound1}
\lim_{T \to \infty} \frac{1}{T} \log \pi_\varphi(T,\alpha,\mathbbm{1}_{[a,b]}) = \beta = P(\varphi+f_\xi)=Q(0).
\end{equation}
Also, for every $\psi \in C(M,\mathbb R)$, we have
\begin{align*}
\sum_{\gamma \in \mathcal P(\alpha)}
\mathbbm{1}_{[a,b]}(\ell(\gamma)-T)
e^{\int_\gamma (\varphi+\psi)}
&=
e^{-\langle \alpha,\xi \rangle} 
\sum_{\gamma \in \mathcal P(\alpha)}
\mathbbm{1}_{[a,b]}(\ell(\gamma)-T)
e^{\int_\gamma (\varphi + f_\xi+\psi)}
\\
&\le
e^{-\langle \alpha,\xi \rangle} 
\sum_{\gamma \in \mathcal P}
\mathbbm{1}_{[a,b]}(\ell(\gamma)-T)
e^{\int_\gamma (\varphi + f_\xi + \psi)},
\end{align*}
giving
\begin{equation}\label{eq:ld_bound2}
\limsup_{T \to \infty} \frac{1}{T} \log
\sum_{\gamma \in \mathcal P(\alpha)}
\mathbbm{1}_{[a,b]}(\ell(\gamma)-T)
e^{\int_\gamma (\varphi+\psi)} 
\le Q(\psi).
\end{equation}

Now define
\[
\rho := 
\inf_{\nu \in \mathcal K} \sup_{\psi \in C(M,\mathbb R)}
\left(\int \psi \, d\nu -Q(\psi)\right).
\]
Given $\epsilon>0$, it follows from the definition of $\rho$ that for every $\nu \in \mathcal K$, there
exists $\psi \in C(M,\mathbb R)$ such that
\[
\int \psi \, d\nu - Q(\psi) > \rho-\epsilon.
\]
Hence
\[
\mathcal K \subset \bigcup_{\psi \in C(M,\mathbb R)} \left\{\nu \in \mathcal M(X) \hbox{ : }
\int \psi \, d\nu > \rho -\epsilon\right\}.
\]
Since $\mathcal K$ is compact, we can find a finite set of functions
$\psi_1,\ldots,\psi_k \in C(M,\mathbb R)$ such that
\[
\mathcal K \subset \bigcup_{i=1}^k \left\{\nu \in \mathcal M(X) \hbox{ : }
\int \psi_i \, d\nu > \rho -\epsilon\right\}.
\]
We then have
\begin{align*}
\Xi_\varphi(T,\alpha,\mathbbm{1}_{[a,b]},\mathcal K)
&\le
\sum_{i=1}^k
\sum_{\substack{\gamma \in \mathcal P(\alpha) \\ \int \psi_i \, d\mu_\gamma -Q(\psi_i) > \rho-\epsilon}}
\mathbbm{1}_{[a,b]}(\ell(\gamma)-T)
e^{ \int_\gamma \varphi}
\\
&\le
\sum_{i=1}^k 
\sum_{\gamma \in \mathcal P}
e^{-\ell(\gamma)(Q(\psi_i)+\rho-\epsilon) + \int_{\gamma} \psi_i}.
\end{align*}
Recalling the bound (\ref{eq:ld_bound2}), we have
\[
\limsup_{T \to \infty} \frac{1}{T} \log \Xi_\varphi(T,\alpha,\mathbbm{1}_{[a,b]},\mathcal K)
\le -\rho +\epsilon.
\]
Since $\epsilon>0$ is arbitrary, we can combine this with (\ref{eq:ld_bound1}) to obtain
\[
\limsup_{T \to \infty} \frac{1}{T} \log 
\left(\frac{\Xi_\varphi(T,\alpha,\mathbbm{1}_{[a,b]},\mathcal K)}{\pi_\varphi(T,\alpha,\mathbbm{1}_{[a,b]})}\right)
\le -\rho -Q(0).
\]

To complete the proof, we show that $\rho+Q(0)>0$. For any measure $\nu \in \mathcal M(X)$,
we have
\begin{align*}
&\sup_{\psi \in C(M,\mathbb R)} 
\left(\int \psi \, d\nu -Q(\psi) + Q(0)\right)
\\
&=
\sup_{\psi \in C(M,\mathbb R)} 
\left(\int \psi \, d\nu - P(\varphi + f_\xi +\psi) + P(\varphi + f_\xi)\right)
\\
&=
\sup_{\psi \in C(M,\mathbb R)} 
\left(\int (\psi - \varphi - f_\xi) \, d\nu -P(\psi) + P(\varphi+f_\xi)\right)
\\
&=
\sup_{\psi \in C(M,\mathbb R)} 
\left(\int \psi \, d\nu - P(\psi)\right) + P(\varphi+f_\xi) - \int (\varphi +f_\xi) \, d\nu
\\
&= - \inf_{\psi \in C(M,\mathbb R)}
\left(P(\psi) - \int \psi \, d\nu\right)  + P(\varphi+f_\xi) - \int (\varphi +f_\xi) \, d\nu
\\
&= 
-h(\nu) + P(\varphi+f_\xi) - \int (\varphi +f_\xi) \, d\nu,
\end{align*}
where the last equality comes from Lemma \ref{entropyusc}.
If $\nu \in \mathcal K$ then $\nu \ne \mu_{\varphi+f_\xi}$ and the uniqueness of equilibrium states gives 
that
\[
-h(\nu) + P(\varphi+f_\xi) - \int (\varphi +f_\xi) \, d\nu >0.
\]
Since, by Lemma \ref{entropyusc}, the map 
\[
\nu \mapsto -h(\nu) + P(\varphi+f_\xi) - \int (\varphi +f_\xi) \, d\nu
\]
is lower semi-continuous on $\mathcal M(X)$ and $\mathcal K$ is compact, we see that 
$\rho +Q(0)>0$, as required.
\end{proof}

We can now obtain the weighted equidistribution theorem for periodic orbits in a homology class.

\begin{theorem} \label{weighted-equi}
Let $X^t : M \to M$ be a homologically full transitive Anosov flow. 
Let $\varphi : M \to \mathbb R$ be H\"older continuous.
Then the measures
\[
\frac{1}{\pi_\varphi(T,\alpha,\mathbbm{1}_{[a,b]})} 
\sum_{\gamma \in \mathcal P(\alpha)} 
\mathbbm{1}_{[a,b]} (\ell(\gamma)-T)
e^{\int_\gamma \varphi} \mu_\gamma
\]
converge weak$^*$ to $\mu_{\varphi+f_\xi}$, as $T \to \infty$,
and the same holds if we replace $[a,b]$ by $(a,b)$, $(a,b]$ or $[a,b)$.
\end{theorem}

\begin{proof}
Let $\psi \in C(M,\mathbb R)$. Given $\epsilon>0$, let $\mathcal K \subset \mathcal M(X)$ be the compact set
\[
\mathcal K = \left\{\nu \in \mathcal M(X) \hbox{ : } \left|\int \psi \, d\mu - 
\int \psi \, d\mu_{\varphi+f_\xi}\right| \ge \epsilon\right\}.
\]
Using Theorem \ref{th:ld}, we have 
\begin{align*}
&\frac{1}{\pi_\varphi(T,\alpha,\mathbbm{1}_{[a,b]})} 
\sum_{\gamma \in \mathcal P(\alpha)} 
\mathbbm{1}_{[a,b]}(\ell(\gamma)-T)
e^{\int_\gamma \varphi}
\int \psi \, d\mu_\gamma
\\
&=
\frac{1}{\pi_\varphi(T,\alpha,\mathbbm{1}_{[a,b]})} 
\sum_{\substack{\gamma \in \mathcal P(\alpha) \\ \mu_\gamma \notin \mathcal K}} 
\mathbbm{1}_{[a,b]}(\ell(\gamma)-T)
e^{\int_\gamma \varphi}
\int \psi \, d\mu_\gamma
+O(e^{-\eta T}),
\end{align*}
for some $\eta>0$.
Since
\begin{align*}
&\frac{1}{\pi_\varphi(T,\alpha,\mathbbm{1}_{[a,b]})} 
\sum_{\substack{\gamma \in \mathcal P(\alpha) \\ \mu_\gamma \notin \mathcal K}} 
\mathbbm{1}_{[a,b]}(\ell(\gamma)-T)
e^{\int_\gamma \varphi}
\int \psi \, d\mu_\gamma
=
(1-O(e^{-\eta T}))\int \psi \, d\mu_{\varphi+f_\xi}
\\
&+
\frac{1}{\pi_\varphi(T,\alpha,\mathbbm{1}_{[a,b]})} 
\sum_{\substack{\gamma \in \mathcal P(\alpha) \\ \mu_\gamma \notin \mathcal K}} 
\mathbbm{1}_{[a,b]}(\ell(\gamma)-T)
e^{\int_\gamma \varphi}
\left(\int \psi \, d\mu_\gamma -\int \psi \, d\mu_{\varphi+f_\xi}\right),
\end{align*}
we see that
\begin{align*}
\int \psi \, d\mu_{\varphi +f_\xi} -\epsilon
&\le 
\liminf_{T \to \infty} \frac{1}{\pi_\varphi(T,\alpha,\mathbbm{1}_{[a,b]})} 
\sum_{\gamma \in \mathcal P(\alpha)} 
\mathbbm{1}_{[a,b]}(\ell(\gamma)-T)
e^{\int_\gamma \varphi}
\int \psi \, d\mu_\gamma
\\
&\le
\limsup_{T \to \infty} 
\frac{1}{\pi_\varphi(T,\alpha,\mathbbm{1}_{[a,b]})} 
\sum_{\gamma \in \mathcal P(\alpha)} 
\mathbbm{1}_{[a,b]}(\ell(\gamma)-T)
e^{\int_\gamma \varphi}
\int \psi \, d\mu_\gamma
\\
&\le \int \psi \, d\mu_{\varphi+f_\xi} + \epsilon.
\end{align*}
Since $\epsilon>0$ is arbitrary, this completes the proof.
\end{proof}

\section{Linking numbers of knots and helicity}\label{section:linking}

In this section and for the remainder of the paper, $M$ will be a smooth closed connected oriented
$3$-manifold.
\subsection{Vector fields and forms}
We briefly recall some background on vector fields and forms.
Let $M$  have a volume form $\Omega$
and let $X$ be a vector field on $M$ generating a flow $X^t : M \to M$. 
The divergence of $X$ is defined by
\[
L_X\Omega = (\mathrm{div} X) \Omega,
\]
where $L_X$ is the Lie derivative.
%\[
%L_X = \lim_{t \to 0} \frac{(X^t)^*\omega -\omega}{t}.
%\]
We say that $X$ is {\it divergence-free} if $\mathrm{div} X$ is identically zero; this is equivalent to 
the flow $X^t$ being volume-preserving.
For the remainder of the section, we will assume that this holds.

For $0 \le p \le 3$, let 
$\mathcal X^p(M)$ denote the space of (smooth) $p$-forms on $M$. We use the notation
$d : \mathcal X^p(M) \to \mathcal X^{p+1}(M)$ for the exterior derivative, and 
 $i_X : \mathcal X^p(M) \to \mathcal X^{p-1}(M)$ for the interior product.
Since $\Omega$ is a $3$-form, we have that
\begin{equation}\label{eq:switch_interior_product}
 i_X \omega \wedge \Omega = \omega \wedge i_X \Omega
 \end{equation}
for any $\omega\in \mathcal X^1(M)$.
 The Lie derivative, exterior derivative and interior product are related by
{\it Cartan's magic formula}
\begin{equation}\label{eq:cartans_magic_formula}
L_X\omega = i_X d\omega + d(i_X\omega).
\end{equation}

%As already used above, by de Rham's Theorem, the real cohomology groups 
%$H^p(M,\mathbb R)$, $0 \le p \le 3$, may be identified with the quotient space
%\[
%\{\mathrm{closed} \ p\mbox{-forms}\}/\{\mathrm{exact} \ p\mbox{-forms}\}.
%\]

The volume form $\Omega$ gives rise to a volume measure $m$
(normalised to be a probability measure).
We have the following key result.

 \begin{lemma}\label{lem:equiv_of_nullhom}
 $i_X\Omega$ is exact if and only if 
 $\Phi_m=0$.
 \end{lemma}
 
 \begin{proof}
 Since $X$ is divergence-free, we have $L_X\Omega=0$. Since $d\Omega=0$, Cartan's magic formula
 (\ref{eq:cartans_magic_formula}) gives
 $d(i_X\Omega)=0$, i.e., $i_X\Omega$ is closed.
 Let $[i_X\Omega] \in H^2(M,\mathbb R)$ be its cohomology class; we claim that
 $[i_X\Omega]$ and $\Phi_m$ are Poincar\'e duals.
 To see this, let $\omega$ be a closed $1$-form, then, by (\ref{eq:switch_interior_product}),
  \[
 i_X \omega \wedge \Omega = \omega \wedge i_X \Omega.
 \]
Thus,
\begin{align}
\langle [i_X\Omega],[\omega]\rangle
&=\int_M \omega \wedge i_X \Omega = \int_M i_X \omega \wedge \Omega 
\\
&= \int_M \omega(X) \, \Omega 
= \int \omega(X) \, dm 
= \langle \Phi_m,[\omega]\rangle
\end{align}
where the first term is the pairing of $H^2(M,\mathbb R)$ and $H^1(M,\mathbb R)$.
Therefore, $[i_X\Omega]=0$ if and only if $\Phi_m=0$. 
 \end{proof}
 
 If $i_X\Omega$ is exact then we say that $X$ is null-homologous. 
 Since $m$ is equal to the equilibrium state for $\varphi^u$, Lemma \ref{lem:equiv_of_nullhom}
 and Proposition \ref{prop:hom_full_eq_state} tell us that if $X$ is a null-homologous 
 Anosov flow then it is
 homologically full. (Note that volume-preserving flows are automatically transitive.)
 Recall the function $f:M\to \R^b$ from the proof of Proposition \ref{strictly_convex_and_finite_min}.
 With respect to a given basis $w_1,\ldots,w_b$ for $H^1(M,\R)$,
the component $f_i$ of $f=(f_1,\ldots,f_b)$ is given by $\omega_i(X)$, 
where $\omega_i$ is a closed $1$-form in the cohomology class $w_i$.
 We have that
 \[
 \nabla \beta_{\varphi^u}(0) = \int f \, d\mu_{\varphi^u}
 =\int f \, dm =0,
 \]
 so $\xi(\varphi^u)=0$.
 Thus, for null-homologous flows, a special case of Theorem \ref{weighted-equi} is the following.

 \begin{theorem} \label{weighted-equi-nullhom}
Let $M$ be a closed $X^t : M \to M$ be a null-homologous volume-preserving Anosov flow
on a closed oriented $3$-manifold.
Then the measures
\[
\frac{1}{\pi_{\varphi^u}(T,\alpha,\mathbbm{1}_{[a,b]})} 
\sum_{\gamma \in \mathcal P(\alpha)} 
\mathbbm{1}_{[a,b]} (\ell(\gamma)-T)
e^{\int_\gamma \varphi^u} \mu_\gamma
\]
converge weak$^*$ to $m$, as $T \to \infty$,
and the same holds if we replace $[a,b]$ by $(a,b)$, $(a,b]$ or $[a,b)$.
\end{theorem}

This implies Theorem \ref{weighted-equi-nullhom-intro} in the introduction.
 
 \medskip
Now suppose that $M$ is equipped with a Riemannian metric $\rho$
(consistent with the volume form
$\Omega$). 
This induces an inner product $\langle \cdot,\cdot \rangle_x$
on each fibre $T_x^*M$ (and its exterior powers)
and hence an inner product 
\[
\langle \langle \omega,\eta \rangle \rangle = \int_M \langle\omega_x,\eta_x\rangle_x \, \Omega
\]
on
$\mathcal X^p(M)$, $0 \le p \le 3$.
The {\it Hodge star} $* : \mathcal X^p(M) \to \mathcal X^{3-p}(M)$ is defined by
\[
\int_M \omega \wedge *\eta = \langle \langle \omega,\eta \rangle \rangle
\]
and
the {\it Laplacian} $\Delta : \mathcal X^p(M) \to \mathcal X^b(M)$ is defined by
$\Delta = d^*d+dd^*$, where $d^*$ is the {\it codifferential} -- the adjoint of $d$  with respect to $\langle\langle\cdot,\cdot\rangle\rangle$. Then a $p$-form $\omega$
is {\it harmonic} if $\Delta \omega=0$ and we let $\mathcal X_{\mathrm{harm}}^p(M)$
denote the space of harmonic $p$-forms. Harmonic forms are closed and, by the Hodge theorem,
the map $\mathcal X_{\mathrm{harm}}^p(M) \to H^p(M,\mathbb R) : \omega \mapsto [\omega]$
is an isomorphism.
Finally, $\mathcal X_{\mathrm{harm}}^p(M)$ is a closed subspace of $\mathcal X^p(M)$ and we let 
$H : \mathcal X^p(M) \to \mathcal X_{\mathrm{harm}}^p(M)$ denote
the orthogonal projection with respect to $\langle \langle \cdot,\cdot \rangle \rangle$.
(Strictly speaking, one should complete $\mathcal X^p(M)$ with respect to the inner
product to get a Hilbert space, in which case $d$ and $\Delta$ become densely defined.)

\subsection{Linking numbers and linking forms}
Let $M$ be a closed oriented $3$-manifold.
A knot is an embedding of $S^1$ in $M$. In the classical situation, where $M = S^3$, we may define 
the linking number of any two disjoint knots, $\gamma$ and $\gamma'$, as follows. 
(The symbols $\gamma, \gamma'$  will refer to the embeddings themselves, but also the images of these embeddings.) Let $S$ be an oriented surface in $M$ whose boundary is $\gamma$.
Then we define the linking number $\mathrm{lk}(\gamma,\gamma') \in \mathbb Z$ to be the algebraic intersection number of 
$\gamma'$ with this $S$.
In this setting (with $S^3$ thought of as the compactification of $\mathbb R^3$), the linking number is also given by the 
\textit{Gauss linking integral}, 
$$
\mathrm{lk}(\gamma,\gamma')=\frac{3}{4\pi}\int_{S^1}\int_{S^1}
\frac{\dot{\gamma}(s)\times \dot{\gamma'}(t)}{\|\gamma(s)-\gamma'(t)\|^3}\cdot (\gamma(s)-\gamma'(t))
\, ds \,dt.
$$
For details on these definitions, see Chapter 5, part D of \cite{rolfsen}. 

We wish to work with more general closed oriented $3$-manifolds
$M$. In this case, we can define the linking number of disjoint knots $\gamma$ and $\gamma'$
provided at least one of them is null-homologous in $H_1(M,\mathbb R)$,
which is equivalent to null-homologous in $H_1(M,\mathbb Q)$. 
Suppose that $\gamma$ is integrally null-homologous; then there exists an oriented surface $S$
whose boundary is $\gamma$ and
 $\mathrm{lk}_{S}(\gamma,\gamma') \in \mathbb Z$ is defined, as above, to be the algebraic intersection number
of $\gamma'$ and $S$. (If $\gamma$ is rationally null-homologous; then there exists $k \ge 1$ such that $\gamma^k$ is integrally 
null-homologous and we obtain a rational linking number for $\gamma$ by division by $k$.)
If $\gamma'$ is also null-homologous, the result is independent of the choice of $S$ and we 
denote it by $\mathrm{lk}(\gamma,\gamma')$.
In particular, if $M$ is a real homology $3$-sphere 
then each pair of disjoint knots has a well-defined rational linking number. 
More generally, if $\gamma'$ is not null-homologous and $S'$ is another choice of surface for $\gamma$ then
\[
\mathrm{lk}_{S'}(\gamma,\gamma') - \mathrm{lk}_S(\gamma,\gamma') = \langle [\gamma'],
c_{S'-S} \rangle,
\]
where $c_{S'-S}$ is the Poincar\'e dual of the the class in $H_2(M,\mathbb Z)$ defined by $S'-S$.
However, one can avoid these choices by defining linking numbers via a linking form (determined by the Riemannian metric).

%It is possible to generalise the Gauss linking integral from $S^3$ to other $3$-manifolds. 
%To do this we introduce \textit{linking forms}.

%\subsection{Linking forms}
%Let $M$ be a closed oriented $3$-manifold.
First, we define a \textit{double form} on $M$ (see Section 7 of \cite{derham}). 
This is essentially a differential form whose coefficients are other differential forms, 
rather than smooth functions. Let $x^1,x^2,x^3$ be local co-ordinates in $U\subset M$, 
and $y^1,y^2,y^3$ in $U'\subset M.$ A differential form $\alpha$ of degree $p$, with 
coefficients which are differential forms of degree $q$, is represented for $x\in U$ by 
$$
\alpha(x)=\sum_{i_1<\ldots <i_p}\alpha_{i_1\ldots i_p}(x) \, dx^{i_1}
\wedge \cdots\wedge  dx^{i_p},
$$ 
where each $\alpha_{i_1\ldots i_p}$ is represented for $y\in U'$ by 
$$
\alpha_{i_1,\ldots i_p}(x,y)
=\sum_{j_1<\ldots < j_q}a_{i_1\ldots i_pj_1\ldots j_q}(x,y) \, dy^{j_1}
\wedge\cdots \wedge \, dy^{j_q},
$$ 
where the $a_{i_1\ldots i_pj_1\ldots j_q}$ are smooth functions $U\times U'\to \R$. When defined in this way, we call $\alpha$ a double form, and write 
$$
\alpha(x,y)
=\sum_{\substack{i_1<\ldots < i_p\\ j_1<\ldots <j_q}}
a_{i_1\ldots i_pj_1\ldots j_q}(x,y)(dx^{i_1}\wedge \cdots
\wedge   dx^{i_p})(dy^{j_1}\wedge\cdots \wedge dy^{j_q}).
$$ 

We integrate double forms by integrating as single forms in $x$ and then in $y$. By this we mean that given chains $c_1,c_2$ in $M$, $\int_{x\in c_1}\alpha(x,y)$ is a single form which can again be integrated as usual. The integral of $\alpha$ over $c_1\times c_2$  is then defined by 
$$
\int_{c_1\times  c_2}\alpha=\int_{y\in c_2}\int_{x\in c_1}\alpha(x,y).
$$ 

We will be working with forms where $p=q=1$. In this case, $\alpha$ is called a $(1,1)$-form, and is written 
$$
\alpha(x,y)=\sum_{ij}a_{ij}(x,y) \, dx^i \, dy^j.
$$ 
By the above definition, we have that 
$$
\int_{c_1\times c_2}\alpha
=\int_{S^1}\int_{S^1}\alpha(c_1(s),c_2(t))(\dot{c_1}(s),\dot{c_2}(t)) \, ds \, dt,
$$ 
for any two closed curves $c_1, c_2.$ 

With these preliminaries in place, we define a linking form to be any $(1,1)$-form $L$ on $M$ which satisfies that whenever $\gamma,\gamma'$ are disjoint null-homologous knots,
$$\int_{\gamma\times\gamma'} L=\lk(\gamma,\gamma').$$ Such forms exist for $M$, and one class of examples is given in \cite{kotschick-vogel}
(generalising \cite{vogel} for rational homology $3$-spheres). 

To outline this construction, we must give a brief description of the \textit{Green operator}
on $M$, 
and an associated $(1,1)$-form, known as the \textit{Green kernel}. 
The Green operator $G$ is a differential operator which acts as a partial inverse to the 
Laplacian $\Delta : \mathcal X^1(M) \to \mathcal X^1(M)$. 
Recall that $H$ denotes the orthogonal projection onto the 
space of harmonic $1$-forms,
$\mathcal X_{\mathrm{harm}}^1(M)=\ker(\Delta)$. 
Then
$G:\mathcal{X}^1(M)\rightarrow (\mathcal{X}_{\mathrm{harm}}^1(M))^\perp$ is the unique operator satisfying 
$$
G\circ\Delta=\Delta\circ G=\mathrm{Id}-H\text{ and }G\circ H=0.
$$ 
The existence of $G$ is shown on page 134 of \cite{derham}. 

There is a $(1,1)$-form $g(x,y)$ which is the kernel (in the sense of \cite{derham}, Section 17) of 
$G$, i.e.
$$
G(\omega)(x)=\int_{x\in M} \omega(y) \wedge *_yg(x,y),
$$
where $*_y$ denotes the Hodge star in the $y$ co-ordinate.
The 
linking form,
which we shall henceforth refer to as the \textit{Kotschick--Vogel linking form}, is the $(1,1)$-form
$$
L(x,y):=*_yd_yg(x,y),
$$ 
where $d_y$ denotes the exterior derivative in the $y$ co-ordinate.

The following appears as Proposition 1 in \cite{kotschick-vogel}. 

\begin{proposition}[Kotschick and Vogel, \cite{kotschick-vogel}]\label{prop:linking_form} 
The double form
$L(x,y)$ is a linking form. Furthermore, for every $1$-form $\alpha$, there exists a function 
$h : M \to \mathbb R$ such that
\[
\int_{y \in M} L(x,y) \wedge d\alpha(y) =
\alpha(x) -H(\alpha)(x) + dh(x).
\]
\end{proposition}

For the rest of the paper, we will write
\[
\mathrm{lk}(\gamma,\gamma') = \int_{\gamma \times \gamma'} L
\]
for arbitrary disjoint knots $\gamma,\gamma'$.

\subsection{Helicity}\label{subsec:helicity}
We recall the definition of helicity. 
Let $M$ be a closed oriented $3$-manifold with (normalised) volume
form $\Omega$. Let $X$ be a divergence-free vector field on $M$ with
associated volume-preserving flow $X^t: M \to M$.
We assume that $X$ is null-homologous, i.e. that $i_X\Omega$ is exact.
Hence, there exists a $1$-form $\alpha$,
called the \textit{form potential} of $X$, such that $i_X\Omega = d\alpha$.
The helicity $\mathcal H(X)$
of $X$ is then defined by 
$$
\mathcal H(X)=\int_M\alpha\wedge i_X\Omega.
$$ 
The form potential is defined up to the addition of a closed $1$-form but, by Lemma
\ref{lem:equiv_of_nullhom},
\[
\int_M \omega \wedge i_X\Omega = \int_M \omega(X) \, \Omega =0,
\]
so the helicity is independent of this choice.

A convenient way 
of evaluating helicity is given by the \textit{musical isomorphisms} in Riemannian geometry.  
Given a point $x\in M$, each $u\in T_xM$ has a corresponding covector 
$u^\flat\in T_x^*M$ determined uniquely by $u^\flat(w)=\rho(u,w)$ for all $w\in T_xM$. 
Letting $\hat{\rho}_x=(\rho_{ij}(x))_{ij}$, the matrix of metric components at $x$, 
we have that $u^\flat=(\hat{\rho}_xu)^T.$ As $\hat{\rho}_x$ is invertible, with inverse 
$\hat{\rho}_x^{-1}=(\rho^{ij}(x))_{ij}$, we have $u=\hat{\rho}_x^{-1}(u^\flat)^T$, 
and we say $u=(u^\flat)^\sharp$. 
We will use raised indices for the components of a tangent vector and lowered indices for the components of its corresponding covector, meaning $u^i$ is the $i$th component of $u$, and $u_i$ is $(\g_xu)^i.$ With this, we can define the curl of a vector field $Z$ to be the unique vector field with $Z^\flat$ as a form potential. That is, 
$$
i_{\curl Z} \Omega=d(Z^\flat).
$$ 
Thus the components of the curl are given by 
$$
(\curl Z(y))^l=(-1)^{l+1}\left(\frac{\partial Z_{j(l)}}{\partial y^{k(l)}}(y)
-\frac{\partial Z_{k(l)}}{\partial y^{j(l)}}(y)\right),
$$ 
where $k(l)<j(l)$ and $\{k(l),j(l)\}=\{1,2,3\}\setminus\{l\}.$ 

When $\alpha$ is a form potential for $X$, $\curl(\alpha^\sharp)=X$, and we call $\alpha^\sharp$ a \textit{vector potential} for $X$. We see that 
$$
\mathcal H(X)=\int_M\alpha\wedge d\alpha=\int_M\alpha(\curl(\alpha^\sharp))\, \Omega
=\int_M\rho(X,\alpha^\sharp) \, \Omega.
$$ 
Thus the helicity can be thought of as a scalar product $\<X,\curl^{-1}X\>$ of $X$ and its potential field. Here the curl is not invertible, but we abuse notation due to independence of the preimage choice.

\section{Periodic orbits and helicity}\label{section:per_orbits_helicity}

We restrict now to the case where $X^t : M \to M$ is a homologically full transitive 
Anosov flow on a closed oriented $3$-manifold and consider its periodic orbits as knots.
It is interesting to ask how they link with other periodic orbits as the period increases. 
As discussed above, to define a linking number, at least one of the knots needs to be null-homologous
in $H_1(M,\mathbb R)$.
We now consider the sets of periodic orbits $\mathcal P_T(0)$ and $\mathcal P_{T+1}$ 
defined in the introduction.
Since $\mathcal P_T(0)$ and $\mathcal P_{T+1}$ are disjoint, and $\mathcal P_T(0)$ consists of null-homologous orbits, the linking number of a
pair of periodic orbits from these two collections are well-defined. 

\begin{remark}\label{remark:formulation}
The choice of intervals
$(T-1,T]$ and $(T,T+1]$ is somewhat arbitrary. The results below will still hold
if we replace them with $[T+a,T+b]$ and $[T+a',T+b']$, for any $a<b$ and $a'<b'$, provided
$[a,b]$ and $[a',b']$ are disjoint, and we can replace any $[$ with $($ and any $]$ with $)$.
Alternatively, we could use the same interval for each pair and discard pairs which match an
orbit with itself. The latter formulation requires a slightly more involved analysis in Section 9
but the approach used there carries through.
\end{remark}

We define average linking numbers
over the sets of orbits $\P_{T}(0)$ and $\mathcal P_{T+1}$, 
weighted by a H\"older continuous function 
$\varphi :M \to \R$, by
\[
\mathscr L_\varphi(T):=\frac{\displaystyle \sum_{\gamma\in \P_{T}(0),\gamma'\in \mathcal P_{T+1}}
\frac{\mathrm{lk}(\gamma,\gamma')}{\ell(\gamma) \ell(\gamma')}
\exp\left(\int_\gamma \varphi+\int_{\gamma'} \varphi\right)}
{\displaystyle \sum_{\gamma\in \P_{T}(0),\gamma'\in \mathcal P_{T+1}}
\exp\left(\int_\gamma \varphi+\int_{\gamma'} \varphi\right)}
\]
(It will become clear from our results that dividing by the periods of the orbits gives the correct normalisation.)

By Proposition \ref{prop:linking_form}, the linking number of two periodic orbits $\gamma,$ $\gamma'$ as above is given by 
$$
\frac{\mathrm{lk}(\gamma,\gamma')}{\ell(\gamma)\ell(\gamma')}=\int L(x,y)(X(x),X(y))\, d(\mu_\gamma\times\mu_{\gamma'})(x,y).
$$ 
So that we can consider integrals of functions, rather than forms, 
we define $\Lambda: (M\times M)\setminus \Delta(M) \rightarrow \R$ by 
$$
\Lambda(x,y):=L(x,y)(X(x),X(y)),
$$  
where \[
\Delta(M) =\{(x,x) \in M \times M \hbox{ : } x \in M\}
\]
is the diagonal in $M\times M$. 

We recall that if $b = \dim H_1(M,\mathbb R) \ge 1$ then $\xi(\varphi) \in H^1(M,\mathbb R)$
and the function $f_{\xi(\varphi)}$ are defined in Section \ref{section:therm-yoga}.
If $b=0$, we set $\xi(\varphi)=0$ and $f_{\xi(\varphi)}=0$. Our main result is the following.

\begin{theorem}\label{thm:main-homfull} 
Let $X^t : M \to M$ be a homologically full transitive Anosov flow on a closed oriented $3$-manifold
and let
$\varphi : M \to \mathbb R$ be a H\"older continuous function.
Then
\[
\lim_{T \to \infty} \mathscr L_\varphi(T) = 
\int \Lambda \, d(\mu_{\varphi + f_{\xi(\varphi)}} \times \mu_{\varphi}).
\]
\end{theorem}

As a consequence, for null-homologous volume-preserving flows,
we can obtain the helicity $\mathcal{H}(X)$ as the limit of appropriately weighted 
averages of linking numbers.

\begin{theorem}\label{thm:hel-homfull}
Let $X^t : M \to M$ be a null-homologous volume-preserving Anosov flow on a 
closed oriented $3$-manifold.
Then 
$$
\mathcal H(X) 
= \lim_{T \to \infty} \mathscr L_{\varphi^{u}}(T).
$$ 
\end{theorem}

We will prove Theorem \ref{thm:hel-homfull} assuming we have proved Theorem \ref{thm:main-homfull}. The proof of Theorem \ref{thm:main-homfull} appears in the next section.

\begin{proof}[Proof of Theorem \ref{thm:hel-homfull}]
Since $X^t$ is volume-preserving, $\mu_{\varphi^{u}}=m$ and,
since $X$ is null-homologous, $\Phi_m=0$.
We then have, 
\[
\nabla \beta_\varphi(0) = \int f \, d\mu_{\varphi^{u}} = \int f \, dm = 0,
\]
where, again, $f : M \to \R^b$ is the function defined in the proof of Proposition \ref{strictly_convex_and_finite_min},
so $\xi(\varphi^{u})=0$.
Hence we can apply Theorem \ref{thm:main-homfull} to conclude that
$\mathscr L_{\varphi^{u}}(T)$ converges to $ \int \Lambda \, d(m \times m)$.
To complete the proof, we need to show that this integral is equal to the helicity $\mathcal H(X)$.

Let $\alpha$ be a $1$-form such that $d\alpha = i_X\Omega$. Then, using 
Proposition \ref{prop:linking_form},
\begin{align*}
\mathcal H(X) &= \int_M \alpha \wedge i_X\Omega
\\
&=\int_{x \in M} \left(\left(\int_{y \in M} L(x,y) \wedge i_X\Omega(y)\right)
+ H(\alpha)(x) - dh(x)\right) \wedge i_X\Omega(x).
\end{align*}
Now, since $\Phi_m=0$, we have both
\[
\int_{M} H(\alpha) \wedge i_X\Omega = \int H(\alpha)(X(x))
\, dm(x) =0,
\]
and
\[
\int_{M} dh \wedge i_X\Omega =\int dh(X(x))\,dm(x)=0.\]
Therefore
\begin{align*}
\mathcal H(X) &=\int_{x \in M} \left(\int_{y \in M} L(x,y) \wedge i_X\Omega(y)\right)
\wedge i_X\Omega
\\
&= \int_{(x,y) \in M \times M} L(x,y)(X(x),X(y)) \wedge \Omega(x) \wedge \Omega(y)
\\
&= \int_{(x,y) \in M \times M} L(x,y)(X(x),X(y)) \, d(m \times m)
\\
&= \int \Lambda \, d(m \times m),
\end{align*}
as required.
\end{proof}

Another consequence of Theorem \ref{thm:main-homfull} is the following result for the geodesic
flows discussed in Example \ref{ex:orbifold}, stated as
Theorem \ref{thm:genus-zero-orbifold-intro} in the introduction.

\begin{theorem}\label{thm:genus-zero-orbifold}
Let $X^t : M \to M$
be the geodesic flow over a genus zero hyperbolic orbifold. Then
\[
\mathcal H(X)
=
\lim_{T \to \infty}
\frac{1}{\#\mathcal P_T \ \#\mathcal P_{T+1}}
\sum_{\gamma\in \P_{T},\gamma'\in \mathcal P_{T+1}}
\frac{\mathrm{lk}(\gamma,\gamma')}{\ell(\gamma) \ell(\gamma')}.
\]
\end{theorem}

\begin{proof}[Proof of Theorem \ref{thm:genus-zero-orbifold}]
Applying Theorem \ref{thm:main-homfull} with $\varphi =0$, we get that the limit
converges to $\int \Lambda \, d(\mu_0 \times \mu_0)$. However, for the geodesic flows
considered here, the measure of maximal entropy is equal to the volume $m$. Hence the limit is
$\int \Lambda \, d(m \times m) = \mathcal H(X)$, as shown in the proof of Theorem
\ref{thm:hel-homfull}.
\end{proof}

\begin{remark}
For comparison we state a result of Contreras \cite{contreras} which motivated this work.
Contreras studied the asymptotic linking of periodic orbits (without weightings)
for hyperbolic flows on basic sets 
of Axiom A flows on $S^3$.
(We note that $S^3$ does not support Anosov flows.) 
The result of Contreras is
that the average linking number of periodic orbits for $X^t$ restricted to a non-trivial
basic set satisfies 
$$
\lim_{T\rightarrow \infty}\mathscr L_0(T)=\int \Lambda \, d(\mu_0\times \mu_0).
$$

In this setting, there is an explicit formula for $\Lambda(x,y)$, 
$$
\Lambda(x,y)=\frac{3}{4\pi}\frac{X(x)\times X(y)}{\|x-y\|^3}\cdot (x-y),
$$ 
which resembles the integrand of Gauss' linking integral. The above result is proved by comparing this integrand to the distance $\|x-y\|$, and using this comparison to show it is integrable with respect to the orbital measures. The equidistribution of these orbits can then be exploited to complete the proof. We will follow a similar approach to prove Theorem \ref{thm:main-homfull}. 
\end{remark}

\section{Proof of Theorem \ref{thm:main-homfull}}\label{section:proof_of_main}

\subsection{Bounds on the Kotschick--Vogel linking form}
For the background to this subsection, see Sections 27 and 28 of \cite{derham}. 
Denote the Riemannian distance on $M$ by $r(x,y)$. 
We will be interested in estimates on $L(x,y)$ as this distance tends to zero.
A $(1,1)$-form is said to be $\O(r^k)$ if its coefficients are. Since the Green kernel $g(x,y)$ (and thus $L(x,y)$) is smooth away from $\Delta(M)$, we will mainly be concerned with the behaviour of the linking form near $\Delta(M).$

From \cite{derham} (page 133), 
\begin{equation}\label{eqn:kernelparametrix}
g(x,y)=\omega(x,y)+\O(r),
\end{equation}
where $\omega(x,y)$ is the \textit{parametrix}, which after setting
$A(x,y)=-\frac{1}{2}r(x,y)^2$, is
defined by 
$$
\omega(x,y) =\frac{1}{s_3}\sum_{ij}\frac{1}{r}\frac{\partial^2A}{\partial x^i\partial y^j}dx^idy^j,
$$
where $s_3$ is the volume of the 3-dimensional unit sphere. Note that in the above we have omitted evaluating functions at $(x,y)$. We will continue in this way to avoid cumbersome notation where possible. 

We now state some properties of the geodesic distance that are useful when working with the parametrix. Let $\xi\in T_xM$ denote the tangent at $x$ to the geodesic from $x$ to $y$, such that $\exp_x(\xi)=y.$ Let $-\eta\in T_yM$ be the tangent at $y$ satisfying $\exp_y(-\eta)=x$.
The functions $\xi(x,y)$, $\eta(x,y)$, $A(x,y)$, as well as the local co-ordinate functions $x^i, y^i$, and their relations are studied in \cite{derham} (page 115). We summarise the results in the following lemma.

\begin{lemma}\label{lemma:dist_properties} The functions above satisfy the following, for all pairs $i,j\in \{1,2,3\}.$
\begin{enumerate}
    \item[(i)] $\xi^i=\O(r)$, $\eta^i=\O(r)$
    \item[(ii)] $y^i-x^i=\O(r)$, $y^i-x^i-\xi^i=\O(r^2)$, \text{ and } $x^i-y^i-\eta^i=\O(r^2).$
    \item[(iii)] $\displaystyle\frac{\partial\xi^i}{\partial x^j}=-\delta_j^i+\O(r)$, and $\displaystyle\frac{\partial\xi^i}{\partial y^j}=\delta_j^i+\O(r),$
    \item[(iv)] $\displaystyle\frac{\partial A}{\partial x^i}=\xi_i$, and $\displaystyle\frac{\partial A}{\partial y^i}=\eta_i.$
\end{enumerate}
Here $\delta^i_j$ is the Kronecker delta symbol.
\end{lemma}

The remainder of this subsection is dedicated to proving the following.

\begin{lemma}\label{lemma:I=O(1/r)}
There exists $K>0$ such that for all $x,y\in M$, 
\[
|\Lambda(x,y)|<\frac{K}{r(x,y)}.
\]
\end{lemma}
By (\ref{eqn:kernelparametrix}) it suffices to show that $*_yd_y\o(x,y)(X(x),X(y))=\O(1/r)$. To understand the coefficients of the (1,1)-form $*_yd_y\o(x,y),$ we will use the following geometric identity. Given a vector field $Z$ on $M$,
\begin{equation}\label{eqn:curlidentity}
*d(Z^\flat)=(\curl Z)^\flat.
\end{equation}
With this identity in mind, we are ready to consider bounding $*_yd_y\o(x,y).$
First, taking $x,y$ close enough, we can assume they are in the same co-ordinate chart, 
meaning that they also have the same metric components. From now on, we will write 
$\rho_{ij}$ instead of $\rho_{ij}(x)=\rho_{ij}(y)$, and $\g$ instead of $\g_x$, for notational ease. 
All sums in what follows are taken over indices ranging from 1 to 3, unless otherwise specified.

We wish to apply (\ref{eqn:curlidentity}) to the $y$-part of the (1,1)-form $\o(x,y).$ Given $x\in M$, let 
$$
\alpha_j^{i,x}(y)=\frac{1}{r(x,y)}\frac{\partial \xi_i}{\partial y^j}(y),\text{ and }\alpha^{i,x}(y)=\sum_j\alpha_j^{i,x}(y) \, dy^j.
$$
Then  $\o(x,y)=\frac{1}{s_3}\sum_i\alpha^{i,x}(y) \, dx^i$, and 
\begin{align*}
    s_3*_yd_y\o(x,y)&=\sum_i(*d\alpha^{i,x})(y)dx^i=\sum_i(\curl (\alpha^{i,x})^\sharp)^\flat dx^i\\
    &=\sum_i\left(\sum_l(-1)^{l+1}\left(\frac{\partial \alpha^{i,x}_j}{\partial y^k}-\frac{\partial \alpha^{i,x}_k}{\partial y^j}\right)\frac{\partial}{\partial y^l}\right)^\flat dx^i\\
    &=\sum_i\left(\sum_l(-1)^{l+1}\left(\frac{\partial \xi_i}{\partial y^j}\frac{\partial r^{-1}}{\partial y^k}-\frac{\partial \xi_i}{\partial y^j}\frac{\partial r^{-1}}{\partial y^j}\right)\frac{\partial}{\partial y^l}\right)^\flat dx^i\\
    &=\sum_{il\lambda}\rho_{l\lambda}(-1)^{l+1}\left(\frac{\partial \xi_i}{\partial y^j}\frac{\partial r^{-1}}{\partial y^k}-\frac{\partial \xi_i}{\partial y^j}\frac{\partial r^{-1}}{\partial y^j}\right) dx^i dy^\l\\
    &=\sum_{il\lambda}\rho_{l\lambda}(\nabla_yr^{-1}\times \nabla_y\xi_i)^ldx^i dy^\l,
\end{align*}
where $k=k(l)<j(l)=j$ are such that $\{k(l),j(l)\}=\{1,2,3\}\setminus\{l\}.$ Using the component formulae for the musical isomorphism $\flat$ and standard rules for the cross and dot product in $\R^3$, we obtain 
\begin{align*}
    s_3*_yd_y\o(x,y)(X(x),X(y))&=\sum_{il\lambda}\rho_{l\lambda}(\nabla_yr^{-1}\times \nabla_y\xi_i)^lX(x)^iX(y)^\l\\
    &=\sum_{il}X(y)_l(\nabla_yr^{-1}\times \nabla_y\xi_i)^lX(x)^i\\
    &=\sum_{ilp}X(y)_l\rho_{ip}(\nabla_yr^{-1}\times \nabla_y\xi^p)^lX(x)^i\\
    &=\sum_{lp}X(y)_l(\nabla_yr^{-1}\times X(x)_p\nabla_y\xi^p)^l\\
    &=(X(y)^\flat)^T\cdot \left(\nabla_yr^{-1}\times \sum_pX(x)_p\nabla_y\xi^p\right)\\
    &=-\nabla_yr^{-1}\cdot \left(\g X(y)\times \sum_pX(x)_p\nabla_y\xi^p\right).
\end{align*}
Applying Lemma \ref{lemma:dist_properties}, we have
\begin{align*}
    \nabla_yr^{-1}&=-\frac{1}{r^2}\nabla_yr=-\frac{1}{r^3}r\nabla_yr=-\frac{1}{2r^3}\nabla_yr^2\\
    &=\frac{1}{r^3}\nabla_yA=\frac{1}{r^3}(\eta^\flat)^T=\frac{1}{r^3}\g\eta.
\end{align*}
Recalling that $\frac{\partial\xi^p}{\partial y^j}=\delta^p_j+\O(r),$ we have
$$
\sum_pX(x)_p\nabla_y\xi^p=(X(x)^\flat)^T+\O(r)=\g X(x)+\O(r).
$$
By Taylor's theorem applied to $X$, for $x$ and $y$ sufficiently close, this gives
$$
\sum_pX(x)_p\nabla_y\xi^p=\g X(y)+\O(r).
$$

Collecting terms from above, we obtain
$$
*_yd_y\o(x,y)(X(x),X(y))=-\frac{1}{r^3}\g\eta\cdot (\g X(y)\times \O(r)).
$$
Since $\eta=\O(r)$ and $\g$ and $X$ are bounded, Lemma \ref{lemma:I=O(1/r)} is proved.

\subsection{Proof of the main result}
We are now ready to prove Theorem \ref{thm:main-homfull}. We will follow the method in \cite{contreras}, aided by Lemma \ref{lemma:I=O(1/r)}, Theorem \ref{weighted-nonnull-equi}
and Theorem \ref{weighted-equi}. 

We define measures
\[
\mu_{\varphi,T} = \frac{\displaystyle \sum_{\gamma \in \mathcal P_{T}} e^{\int_\gamma \varphi} \mu_\gamma}{
\displaystyle \sum_{\gamma \in \mathcal P_{T}}
e^{\int_\gamma \varphi}},
\hspace{24pt}
\mu_{\varphi,T}^0 = \frac{\displaystyle \sum_{\gamma \in \mathcal P_T(0)}
e^{\int_\gamma \varphi} \mu_\gamma}{
\displaystyle \sum_{\gamma \in \mathcal P_T(0)} e^{\int_\gamma \varphi}}.
\]
By Theorem \ref{weighted-nonnull-equi}, $\mu_{\varphi,T+1}$ converges to the equilibrium 
state $\mu_\varphi$. If $M$ is a real homology $3$-sphere then $\mu_{\varphi,T}^0$ also converges to $\mu_\varphi$. On the other hand, if $H_1(M,\mathbb R)$ has dimension at least one then,
by Theorem \ref{weighted-equi}, $\mu_{\varphi,T}^0$ converges to the equilibrium state
$\mu_{\varphi+f_{\xi(\varphi)}}$.
To simplify notation, we shall write
\[
\varphi^* = \begin{cases}
\varphi \qquad \quad \ \mbox{ if } \dim H_1(M,\mathbb R)=0
\\
\varphi + f_{\xi(\varphi)}  \mbox{ if } \dim H_1(M,\mathbb R) \ge 1,
\end{cases}
\]
so that the limit of $\mu_{\varphi,T}^0$ is denoted 
$\mu_{\varphi^*}$ in all cases.
By definition
$$
\mathscr{L}_\varphi(T)=\int\Lambda \, d(\mu_{\varphi,T}^0\times \mu_{\varphi,T+1}).
$$ 
Therefore, if 
either of the following limits exist, then we have the equality
\begin{equation}\label{eq:equal_limits} 
\lim_{T \to \infty} \mathscr L_\varphi(T)
=\lim_{T\rightarrow \infty}\int\Lambda \, d(\mu_{\varphi,T}^0\times\mu_{\varphi,T+1}).
\end{equation}

We have that $\mu_{\varphi,T}^0$ converges to $\mu_{\varphi^*}$ and $\mu_{\varphi,T+1}$ converges to $\mu_\varphi$. We will use this to prove that integral in (\ref{eq:equal_limits}) converges to the integral over 
$\mu_{\varphi^*} \times \mu_{\varphi}$. First we must prove 
$\int \Lambda \, d(\mu_{\varphi^*}\times \mu_{\varphi})$ exists.

We will use the following lemma from \cite{contreras}.

\begin{lemma}[\cite{contreras}, Lemma 2.4]\label{lemma:O(1/r)_integrable}
Let $(Y,d)$ be a separable metric space with Borel probability measures $\mu$
and $\nu$.
\begin{enumerate}
    \item[(i)] 
    If $x\in Y$ is such that $\liminf_{\rho\rightarrow 0}\frac{\log \mu(B(x,\rho))}{\log \rho}>1$, then $\int \frac{1}{r(x,a)}\, d\mu(a)$ exists.
    \item[(ii)] 
    If this limit is uniformly greater than $1$ $\nu$-a.e, then $\int \frac{1}{r(x,a)}\, d(\mu(x)\times\nu(a))$ exists.
\end{enumerate}
\end{lemma}

By Lemmas \ref{lemma:I=O(1/r)} and \ref{lemma:O(1/r)_integrable}, to prove 
$\int \Lambda \, d(\mu_{\varphi^*} \times \mu_{\varphi})$ exists we need only show 
$$
\liminf_{\rho\rightarrow 0}\frac{\log \mu_{\varphi^*}(B(x,\rho))}{\log \rho}>1
$$ 
uniformly $\mu_{\varphi}$-almost everywhere. To do this, we will use a bound on the measure $\mu_{\varphi^*}$ which resembles a Gibbs property. We will need the following definition.

A function $\psi : M \to \mathbb R$ satisfies the {\it Bowen property} if there exist $
C,\delta >0$ such that for all $L>0$, whenever $y\in B(x,\delta,L),$  
$$
\bigg|\int_0^L \psi(X^t(x)) \, dt-\int_0^L \psi(X^t(y)) \, dt\bigg|<C.
$$
This holds for H\"older continuous functions.

\begin{lemma}[Franco \cite{franco}]\label{lemma:Franco} 
Suppose $\psi$ satisfies the Bowen property, and $\delta>0$ is small. Then, there exists $C_\delta>0$ such that for any $L>0$, $x\in M$,
\[
\mu_{\psi}(B(x,\delta,L))  \le C_\delta\exp{\left(\int_0^L\psi(X^t(x))\, dt-P(\varphi)L\right)}.
\]
\end{lemma}

\begin{proof}
The method we follow is based on that of Franco (\cite{franco}, Proposition 2.11). We first consider bounding the orbital measures $\mu_{\psi, T}(B(x,\delta,L))$, for $T>L$. To obtain bounds involving pressure, we first construct a large separated set of periodic points contained in $B(x,\delta,L)$.

Let $|\mathcal P_{T}|$ denote the set of points on the orbits in $\mathcal P_{T}$. By expansivity, there exists a constant $q>0$ such that for $y,y'\in |\mathcal P_{T}|$, $y'\not\in X^{[-q,q]}y$ implies $y$ and $y'$ are $(T,2\delta)$-separated. Let $\delta'=\min\{q,\delta\}$ and choose an integer $S>2q/\delta'$.  Consider an orbit $\gamma\in \mathcal P_{T}$. We can divide $\gamma$ into consecutive closed segments $I_1,\ldots, I_m$, such that each segment has the same orbit length $l$, for some $l\in (\delta'/2,\delta')$. By the definition of $q$, if $|i-j|>S \,(\text{mod } m),$ we have $I_i\cap X^{[-q,q]}I_j=\varnothing.$ We will now distribute the segments into collections $E_1,\ldots, E_{2(S+1)}$ such that if $I_i,I_j\in E_k$ are distinct, then $|i-j|>S \,(\text{mod } m).$ We do this with the following process:
\begin{enumerate}
\item Put $I_1\in E_1$, then add the $\lfloor \tfrac{m}{S+1} \rfloor -1$ other segments $I_{S+2}, I_{2S+3},I_{3S+4},\ldots$
\item Put $I_2\in E_2$, and the $\lfloor \tfrac{m}{S+1} \rfloor -1$ other segments $I_{S+3}, I_{2S+4}, I_{3S+5},\ldots$
\item Repeat this process until collections $E_1,\ldots, E_{S+1}$ are full, at which point at most $S+1$ segments remain.
\item Put each of the remaining segments (if any) into a collection on its own, and leave the remaining collections (if any) empty.
\end{enumerate}
Now, we have that 
$$\mu_{\psi, T}(B(x,\delta,L)\cap\gamma)=\sum_{k=1}^{2(S+1)}\mu_{\psi, T}\left(B(x,\delta,L)\cap \bigcup_{I_i\in E_k}I_i\right),$$
 so there exists $k^*$ such that $E_{k^*}$ satisfies
$$\mu_{\psi, T}\left(B(x,\delta,L)\cap \bigcup_{I_i\in E_{k^*}}I_i\right)\geq \frac{1}{2(S+1)}\mu_{\psi, T}(B(x,\delta,L)\cap\gamma).$$
 Form a set $A_\gamma$ by picking one point (wherever possible) from $B(x,\delta,L)\cap I_i$, for each $I_i\in E_{k^*}$. For $y\in A_\gamma\cap I_i$, set $$R_y=\{t\in (-\ell(\gamma),\ell(\gamma)) :X^ty\in B(x,\delta,L)\cap I_i\}\subset [-\delta,\delta].$$ If we let $A=\cup_{\gamma\in \mathcal P_{T+1}}A_\gamma$, then $A$ is $(T,2\delta)$-separated and we have 
$$\mu_{\psi, T}(B(x,\delta,L))\leq 2(S+1)\mu_{\psi, T}\left(\bigcup_{y\in A}X^{R_y}y\right)=2(S+1)\frac{\sum_{y\in A}\lambda(R_y)e^{\int_{\gamma_y}\psi}}{\sum_{\gamma \in \mathcal P_{T+1}}e^{\int_\gamma \psi}},$$ where $\lambda$ is Lebesgue measure on the real line, and $\gamma_y$ refers to the periodic orbit containing $y$. Now, since $A\subset B(x,\delta, L)$, $X^LA$ is $(T-L,2\delta)$-separated and by the Bowen property there is $C>0$ such that for each $y\in A$, 
$$\bigg|\int_0^L \psi(X^t(x)) \, dt-\int_0^L \psi(X^t(y)) \, dt\bigg|<C.$$
Thus $$\frac{\sum_{y\in A}\lambda(R_y)e^{\int_{\gamma_y}\psi}}{\sum_{\gamma \in \mathcal P_{T+1}}e^{\int_\gamma \psi}}\leq 2\delta e^{C+\int_0^L\psi(X^tx)\,dt}\frac{\sum_{y\in X^LA}e^{\int_0^{T-L}\psi(X^ty)\, dt}}{\sum_{\gamma \in \mathcal P_{T+1}}e^{\int_\gamma \psi}}.$$ Combining the above with Lemmas 2.6 and 2.8 in \cite{franco}, we have $C_\delta>0$ such that $$\mu_{\psi, T}(B(x,\delta,L))\leq C_\delta\exp{\left(\int_0^L\psi(X^t(x))\, dt-P(\varphi)L\right)}.$$ By Theorem \ref{weighted-nonnull-equi},  $$\mu_\psi(B(x,\delta,L))\leq \liminf_{T\to\infty}\mu_{\psi, T}(B(x,\delta,L)),$$ which completes the proof.
\end{proof}

\begin{remark}
The results from \cite{franco} that are used in the proof above are only proved for periodic orbits whose least period lies within a small range $(T-\epsilon,T+\epsilon),$ as opposed to our range of $(T,T+1]$. Strictly speaking, one should obtain the result of Theorem \ref{thm:main-homfull} for these smaller ranges and then apply an additive argument to the limit in order to obtain it for the larger range.
\end{remark}

\begin{lemma}\label{lemma:I_integrable} Suppose $\psi$ satisfies the Bowen property. Then 
$$
\liminf_{\rho\rightarrow 0}\frac{\log \mu_{\psi}(B(x,\rho))}{\log \rho}>1
$$ 
uniformly in $x$. 
\end{lemma}

\begin{proof} We follow the proof of Lemma 2.6 in \cite{contreras}. 
By Lemma \ref{lem:coho-to-neg}, 
there exists $\epsilon>0$ and a H\"older continuous function $v : M \to \mathbb R$ such that
\[
\int_0^L \psi(X^tx)\, dt -P(\psi)L\le -\epsilon L + v(X^Lx)-v(x),
\]
for all $x \in M$ and $L \ge 0$.
Thus, the proof of Lemma \ref{lemma:Franco} 
tells us that  
for $\delta>0$ sufficiently small, 
$$
\mu_{\psi, T}(B(x,\delta,L))\leq C_\delta^\prime e^{-\epsilon L},
$$
where $C_\delta^\prime = C_\delta e^{2\|v\|_\infty}$.

By compactness of $M$, there exist constants $\lambda, k_1,k_2>0$ such that $\|D_xX^t\|\leq \lambda^t$ for $t\geq 0$, and $k_1<\|X\|<k_2$. 
So whenever $\rho\lambda^L\leq \delta/2,$ $B(x,\rho)\subset B(x,\delta/2,L)$, 
and, for any $a>0,$ 
\[
X^{[-a,a]}B(x,\rho)\subset B(x,\delta/2+ak_2,L).
\] 
Now, any orbit intersecting $X^{[-a,a]}B(x,\rho)$ does so for time at least $2a$ and
any orbit intersecting $B(x,\rho)$ does so for time at most $2\rho/k_1.$ 
Setting $a=\delta/2k_2,$ 
$$
\mu_{\psi, T}(B(x,\rho))\leq \frac{2\rho}{2ak_1}
\mu_{\psi, T}(X^{[-a,a]}B(x,\rho))
\leq \frac{\rho}{ak_1}\mu_{\psi, T}(B(x,\delta,L)).
$$ 
Since $B(x,\rho)$ is open and $\mu_{\psi, T}\rightarrow \mu_{\psi}$, we have
\begin{align*}
\mu_{\psi}(B(x,\rho))
&\leq\liminf_{T \rightarrow\infty}\mu_{\psi, T}(B(x,\rho))
\\
&\le \frac{\rho}{ak_1} \liminf_{T \to \infty} \mu_{\psi, T}(B(x,\delta,L))
\leq \frac{C_\delta^\prime \rho}{ak_1}e^{-\epsilon L}.
\end{align*}
Now, if we set $\rho>0$ sufficiently small and consider 
$L=L(\rho):=\frac{\log\delta/2-\log\rho}{\log\lambda},$ then  $\rho\leq \delta/2\lambda^L$, and so 
$$
\frac{\log \mu_{\psi}(B(x,\rho))}{\log \rho}\geq 1+\frac{\epsilon}{\log\lambda}
+O\left(\frac{1}{\log (1/\rho)}\right).
$$ 
Thus we have that 
$$
\liminf_{\rho\rightarrow 0}\frac{\log \mu_{\psi}(B(x,\rho))}{\log \rho}\geq 1+\frac{\epsilon}{\log\lambda}>1,
$$
uniformly in $x$.
\end{proof}

Having shown that our integral exists, we are left to show it is the value of the limit on the right-hand side of
(\ref{eq:equal_limits}). We again examine the behaviour of $\Lambda$ near the diagonal. First we show that the diagonal in $M \times M$ has zero measure with respect to the product of two equilibriums states of H\"older continuous functions.

\begin{lemma}\label{lem:diagonal-zero-measure}
Let $\psi^\prime : M \to \mathbb R$, $\psi: M\to \mathbb R$ be H\"older continuous. Then
\[
(\mu_{\psi^\prime} \times \mu_\psi)(\Delta(M))=0.
\]
\end{lemma}

\begin{proof}
 We can cover $\Delta(M)$ by products $B(x,\delta,T) \times B(x,\delta,T)$,
where $x$ runs over a $(T,\delta)$-spanning set $E_{\delta,T}$.  Applying Lemma
\ref{lemma:Franco} followed by Lemma \ref{lem:coho-to-neg}, we have some $\epsilon, k_\delta,k_\delta^\prime>0$ such that
\begin{align*}
(\mu_{\psi^\prime}\times \mu_\psi)(\Delta(M))
&\le
\sum_{x \in E_{\delta,T}} \mu_{\psi^\prime}(B(x,\delta,T))\mu_\psi(B(x,\delta,T))
\\
&\le 
k_\delta\sum_{x \in E_{\delta,T}} \exp\left( \int_0^T \psi^\prime(X^tx) \, dt-P(\psi^\prime)T+ \int_0^T \psi(X^tx) \, dt-P(\psi)T\right)
\\
&\le
k_\delta' \sum_{x \in E_{\delta,T}}
\exp\left( \int_0^T \psi^\prime(X^tx) \, dt-P(\psi^\prime)T -\epsilon T\right),
\end{align*}
Since $E_{\delta,T}$ was an arbitrary
$(T,\delta)$-spanning set, 
$$
(\mu_{\psi^\prime}\times \mu_\psi)(\Delta(M))\leq k_\delta'e^{-(P(\psi')+\epsilon)T}\inf\bigg\{\sum_{x \in E}
e^{\int_0^T \psi^\prime(X^tx) \, dt}:E \mbox{ is }(T,\delta)\text{-spanning}\bigg\}.
$$
Provided $\delta>0$ is chosen sufficiently small, we can take $T\to\infty$ and use the topological definition of pressure to conclude that $(\mu_{\psi^\prime} \times \mu_\psi)(\Delta(M))=0$.
\end{proof}

\begin{lemma}\label{lemma:limit_to_diag} If there exists a nested collection 
$\{B_R\}_{0< R\leq R_0}$ of open neighbourhoods of $\Delta(M)$ with $\bigcap_{0<R \le R_0} B_R = \Delta(M)$, such that 
$$
\lim_{R\rightarrow 0}\lim_{T\rightarrow \infty}\int_{B_R}\Lambda \, 
d(\mu_{\varphi,T}^0\times\mu_{\varphi,T+1})=0,
$$ 
then the following limit exists, and equality holds:
$$
\lim_{T\rightarrow \infty}\int \Lambda \, d(\mu_{\varphi,T}^0\times\mu_{\varphi,T+1})=\int \Lambda \, d(\mu_{\varphi^*}\times \mu_{\varphi}).
$$
\end{lemma}

\begin{proof}
Suppose the hypothesis is satisfied. We have that 
$\mu_{\varphi,T}^0\times\mu_{\varphi,T+1}\rightarrow \mu_{\varphi^*} \times \mu_{\varphi}$ 
in the weak$^*$ topology. 
Let $A_R=(M \times M) \setminus B_R$, then as $\Lambda$ is continuous away from the diagonal, 
$$
\lim_{T\rightarrow \infty}\int_{A_R} \Lambda \, d(\mu_{\varphi,T}^0\times\mu_{\varphi,T+1})
=\int_{A_R} \Lambda \, d(\mu_{\varphi^*}\times \mu_{\varphi}).
$$ 
Now, as $\Lambda$ is $(\mu_{\varphi^*} \times \mu_{\varphi})$--integrable and,
by Lemma \ref{lem:diagonal-zero-measure}, 
$(\mu_{\varphi^*}\times \mu_{\varphi})(\Delta(M))=0$, we have
$$
\lim_{R \to 0} \int_{B_R} \Lambda \, d(\mu_{\varphi^*} \times \mu_{\varphi})
=0.
$$ 
These facts, along with the hypothesis and the triangle inequality, yield that, for 
$\delta>0$, there exist $T_0>0$ such that $T>T_0$ implies 
$$
\bigg|\int \Lambda \, d(\mu_{\varphi,T}^0\times\mu_{\varphi,T+1})
-\int \Lambda \, d(\mu_{\varphi^*} \times \mu_{\varphi})\bigg| < \delta.
$$
\end{proof}

Now we must exhibit the sets $B_R$ and show they satisfy the required property. To do this we will consider bounds for the integral of $\Lambda$ with respect to $\mu_{\varphi,T+1}$. Recall the notation of Lemma \ref{lemma:I_integrable}.

\begin{lemma}
There exists $\delta,R,\alpha,Q>0$, independent of $T$, such that for all $x\in M$ and 
$\delta/2\lambda^T\leq R$, 
$$
\int_{B(x,R)\setminus B(x,\delta/2\lambda^T)} |\Lambda(x,y)|\, d\mu_{\varphi,T+1}(y)\leq QR^\alpha.
$$ 
\end{lemma}

\begin{proof} Choose $\delta$ as in the proof of Lemma \ref{lemma:I_integrable} to be smaller than the expansivity constant for $X$. Our calculations there also show that there exists $R>0$ such that whenever $0<\rho<R$
$$
\mu_{\varphi,T+1}(B(x,\rho))\leq \frac{\rho}{ak_1}\mu_{\varphi,T+1}(B(x,\delta,L(\rho))).
$$
Now, by the proof of Lemma \ref{lemma:Franco}, there exists some $C_{\delta}>0$ such that for $T>L(\rho)$, 
\begin{equation}\label{eqn:gibbs-for-orbital}
\mu_{\varphi,T+1}(B(x,\delta,L(\rho)))\leq C_\delta\exp{\left(\int_0^{L(\rho)}\varphi(X^t(x))dt-P(\varphi)L(\rho))\right)}.
\end{equation}
By Lemma \ref{lem:coho-to-neg} we have $\epsilon, K_\delta, K_\delta'>0$ such that  
$$
\mu_{\varphi,T+1}(B(x,\rho))\leq
\frac{K_\delta \rho}{2ak_1}e^{-\epsilon L(\rho)}=K_\delta'\rho^{1+\tfrac{\epsilon}{\log \lambda}}.
$$
Set $\alpha=\epsilon/\log \lambda$, and define $N_T=\min\{n\in \N : R/2^{n}\leq
\delta/2\lambda^T\}$. Let 
$$
A_n(x)=B(x,R/2^{n-1})\setminus B(x,R/2^{n});\text{ for }x\in M\text{ and }1\leq n \leq N_T -1,
$$ 
$$
A_{N_T}(x)=B(x,R/2^{N_T-1})\setminus B(x,\delta/2\lambda^T).
$$
Now, splitting our integral over these annuli, and using Lemma \ref{lemma:I=O(1/r)}, we have
\begin{align*}
\int_{B(x,R)\setminus B(x,\delta/2\lambda^T)} |\Lambda(x,y)|\, d\mu_{\varphi,T+1}(y)&=\sum_{n=1}^{N_T}\int_{A_n(x)} |\Lambda(x,y)|\, d\mu_{\varphi,T+1}(y)\\
&\leq \sum_{n=1}^{N_T}\frac{2^{n}K}{R}\mu_{\varphi,T+1}(B(x,R/2^{n-1}))\\
&\leq 2KK_\delta'R^\alpha\sum_{n=0}^{N_T-1}\frac{1}{2^{\alpha n}}\leq \frac{2KK_\delta'}{1-2^{-\alpha}} \, R^\alpha.
\end{align*}
Setting $Q=2KK_\delta'/(1-2^{-\alpha})$, we are done.
\end{proof}

Now, let 
$$ B_R=\bigcup_{x\in M}(\{x\}\times B(x,R))
\text{ and } 
D=\bigcup_{x\in M}(\{x\}\times B(x,\delta/2\lambda^T)).
$$ 
It is clear that the $B_R$ limit to the diagonal in the required way, so we are done if we show 
the integral limit property in Lemma \ref{lemma:limit_to_diag}. By Fubini's Theorem, 
$$
\int_{B_R\setminus D} |\Lambda|\, d(\mu_{\varphi,T}^0\times\mu_{\varphi,T+1})\leq QR^\alpha.
$$  

It remains to describe the integral on $D$. As 
our measures are supported on periodic orbits, we have 
$$
\int_{D} |\Lambda|\, d(\mu_{\varphi,T}^0\times\mu_{\varphi,T+1})
=\int_{x\in |\P_T(0)|}\int_{y\in B(x,\delta/2\lambda^T)\cap |\P_T^{\prime}|} 
|\Lambda(x,y)|\, d\mu_{\varphi,T}^0(x) d\mu_{\varphi,T+1}(y),
$$ 
where $|\P_T(0)|,|\P_T^{\prime}|$ denote the set of points on orbits in $\P_T(0),\P_T^\prime$ respectively. Now, as we chose $\delta$ smaller than the expansivity constant, we ensure that for any $x\in |\P_T(0)|,$
$$
B(x,\delta/2\lambda^T)\cap |\P_T^{\prime}|=\varnothing.
$$ 
If this were not the case then there would be a distinct periodic orbit from that of $x$ which intersects this ball, whilst also having comparable period to $x$. This violates expansivity. 

So the above integral is zero, and we can conclude 
$$
\lim_{T\rightarrow\infty}\bigg|\int_{B_R} \Lambda\, d(\mu_{\varphi,T}^0\times\mu_{\varphi,T+1})\bigg|
\leq 
QR^\alpha.
$$
Thus
\[
\lim_{R\rightarrow 0}\lim_{T\rightarrow\infty}\bigg|\int_{B_R} \Lambda\, d(\mu_{\varphi,T}^0\times\mu_{\varphi,T+1})\bigg|=0
\]
 as required.

\begin{remark}
It is interesting to ask whether one can obtain versions of Theorem \ref{thm:main-homfull}
and Theorem \ref{thm:hel-homfull} with $\mathcal P_{T+1}$ replaced by
$
\mathcal P_{T+1}(0)$.
One suspects this is the case;
however, we were unable to prove the estimate (\ref{eqn:gibbs-for-orbital}) for the orbital measures corresponding to the null-homologous orbits $\mathcal P_{T+1}(0)$.
\end{remark}

\begin{remark}
As we remarked in the introduction, the helicity is an invariant of volume-preserving diffeomorphisms. 
Arnold \cite{arnold} conjectured that it is invariant under volume-preserving homeomorphisms. 
Unfortunately, our results do not shed any light on this conjecture since a homeomorphism need not
preserve the quantities $\int_\gamma \varphi^u$. Indeed, if they are preserved then the flows are already smoothly conjugate \cite{llave-moriyon}.
For further discussion of this problem, see \cite{muller-spaeth}.
\end{remark}

\end{document}